\documentclass{amsart}
\usepackage{amsfonts,amssymb}
\usepackage{enumerate,graphicx}
\usepackage{caption}
\usepackage{subcaption}
\usepackage[square,numbers]{natbib}
\usepackage{mathrsfs,bbm}
\usepackage{tikz,tikzscale,tikz-cd}
\usepackage{multirow,xparse,fp}
\usepackage{environ}
\usepackage{amstext}
\usepackage{hyperref}
\usepackage{diagbox}
 \usepackage[pagewise]{lineno}
\usepackage{comment}
\newtheorem{theorem}{Theorem}[section]

\newtheorem{corollary}[theorem]{Corollary}
\newtheorem{lemma}[theorem]{Lemma}
\theoremstyle{definition}

\newtheorem{problem}{Problem}
\newtheorem{remark}[theorem]{Remark}

\begin{document}
\title{Large Deviation Principle of Multidimensional Multiple Averages on $\mathbb{N}^d$}

\author[Jung-Chao Ban]{Jung-Chao Ban}
\address[Jung-Chao Ban]{Department of Mathematical Sciences, National Chengchi University, Taipei 11605, Taiwan, ROC.}
\address{Math. Division, National Center for Theoretical Science, National Taiwan University, Taipei 10617, Taiwan. ROC.}
\email{jcban@nccu.edu.tw}

\author[Wen-Guei Hu]{Wen-Guei Hu}
\address[Wen-Guei Hu]{College of Mathematics, Sichuan University, Chengdu, 610064, China}
\email{wghu@scu.edu.cn}

\author[Guan-Yu Lai]{Guan-Yu Lai}
\address[Guan-Yu Lai]{Department of Applied Mathematics, National Chiao Tung University, Hsinchu 30010, Taiwan, ROC.}
\email{guanyu.am04g@g2.nctu.edu.tw}

\keywords{large deviation priciple, multiple sum, weighted multiple sum, free energy function, boundary condition.}

\thanks{Ban is partially supported by the Ministry of Science and Technology, ROC (Contract MOST 109-2115-M-004-002-MY2 and 108-2115-M-004-003). Hu is partially supported by the National Natural Science Foundation of China (Grant 11601355).}

\date{}

\baselineskip=1.2\baselineskip

\begin{abstract}
This paper establishs the large deviation principle (LDP) for multiple averages on $\mathbb{N}^d$. We extend the previous work of [Carinci et al., Indag. Math. 2012] to multidimensional lattice $\mathbb{N}^d$ for $d\geq 2$. The same technique is also applicable to the weighted multiple average launched by Fan [Fan, Adv. Math. 2021]. Finally, the boundary conditions are imposed to the multiple sum and explicit formulae of the energy functions with respect to the boundary conditions are obtained.

\end{abstract}
\maketitle

\section{Introduction}
In this article, we would like to study the large deviation rate function of
the (weighted) multiple average in multidimensional lattice $\mathbb{N}^{d}$%
. Before presenting our main results, we would like to explain below the
motivation behind this study. Let $T:X\rightarrow X$ be a map from $X$ into $%
X$ and $\mathbb{F}=(f_{1},\ldots ,f_{d})$ be a $d$-tuple of functions, where 
$f_{i}:X\rightarrow \mathbb{R}$ for $1\leq i\leq d$. The \emph{%
	multiple ergodic theory} is to study the asymptotic behavior of the
\emph{multiple sum} 
\begin{equation}
A_{n}\mathbb{F}(x)=\sum_{k=0}^{n-1}f_{1}(T_{1}^{k}(x))f_{2}(T_{2}^{k}(x))%
\cdots f_{d}(T_{d}^{k}(x))\text{.}  \label{10}
\end{equation}%
Such problem was initiated by Furstenberg \cite{furstenberg1982ergodic} on
his proof of the Szemer\'{e}di's theorem. Host and Kra \cite%
{host2005nonconventional} proved the $L^{2}$-convergence of (\ref{10}) when $%
T_{j}=T^{j}$ and $f_{j}\in L^{\infty }(\mu )$, Bourgain \cite%
{bourgain1990double} proved the almost everywhere convergence when $d=2$.
Later, the multifractal analysis and the dimension theory of the multiple
ergodic averages $\frac{A_{n}\mathbb{F}(x)}{N}$ in $\mathbb{N}$ (or $\mathbb{Z}$) are also interesting
research subjects and have been studied in depth recently (cf. \cite%
{ban2021entropy,ban2019pattern,
	brunet2020dimensions,fan2012level, fan2016multifractal,  kenyon2012hausdorff,peres2014dimensions, peres2012dimension,
	pollicott2017nonlinear}). We also refer the reader to \cite{fan2014some} for
the survey of this subject and find the complete bibliography therein. Those
works concentrate on what are known as \emph{multiplicative subshifts}.
Precisely, let $\Sigma _{m}=\{0,\ldots ,m-1\}$ and $\Omega \subseteq \Sigma
_{m}^{\mathbb{N}}$ be a subset. Suppose $S$ is a semigroup generated by
primes $p_{1},\ldots ,p_{k-1}$. Set 
\begin{equation}
X_{\Omega }^{(S)}=\{(x_{k})_{k=1}^{\infty }\in \Sigma _{m}^{\mathbb{N}%
}:x|_{iS}\in \Omega ,\forall i\in \mathbb{N}\text{, }\gcd (i,S)=1\},
\label{1}
\end{equation}%
where $\gcd (i,S)=1$ means that $\gcd (i,s)=1$, $\forall s\in S$.
It is worth noting that the investigation of $X_{\Omega }^{(S)}$ was started from
the study of the set $X^{p_{1},p_{2},\ldots ,p_{k-1}}$ defined below.
Namely, if $p_{1},\ldots ,p_{k-1}$ are primes, define 
\begin{equation}
X^{p_{1},p_{2},\ldots ,p_{k-1}}=\{(x_{i})\in \Sigma _{m}^{\mathbb{N}%
}:x_{i}x_{ip_{1}}\cdots x_{ip_{k-1}}=0,\forall i\in \mathbb{N}\}\text{.}
\label{2}
\end{equation}%
It is clear that $X^{p_{1},p_{2},\ldots ,p_{k-1}}$ is a special case of $X_{\Omega }^{(S)}$ with $\Omega $ is a closed subset of $\Sigma _{m}^{\mathbb{N}}$. Results for
Hausdorff and Minkowski dimension of (\ref{1}) or (\ref{2}) are obtained in 
\cite{ban2019pattern,fan2012level, kenyon2012hausdorff,peres2014dimensions}%
. The authors call $X_{\Omega }^{(S)}$ `multiplicative subshifts' in \cite%
{kenyon2012hausdorff} since it is invariant under \emph{multiplicative
	integer action}. That is, 
\[
x=(x_{k})_{k\geq 1}\in X_{\Omega }^{(S)}\text{ }\Rightarrow \text{ }\forall
i\in \mathbb{N}\text{, }(x_{ik})_{k\geq 1}\in X_{\Omega }^{(S)}.
\]

For $\mathbf{p}_{1},\ldots ,\mathbf{p}_{k-1}\in \mathbb{N}^{d}$, 
\begin{equation}
X^{\mathbf{p}_{1}\mathbf{,p}_{2},\ldots ,\mathbf{p}_{k-1}}=\{(x_{\mathbf{i}%
})\in \Sigma _{m}^{\mathbb{N}^{d}}:x_{\mathbf{i}}x_{\mathbf{i\cdot p}%
	_{1}}\cdots x_{\mathbf{i\cdot p}_{k-1}}=0,\forall \mathbf{i}\in \mathbb{N}%
^{d}\}\text{,}  \label{3}
\end{equation}%
where $\mathbf{i\cdot j}$ denotes the standard inner product of $\mathbf{i}$
and $\mathbf{j}$, i.e., $\mathbf{i\cdot j=}\sum_{l=1}^{d}i_{l}j_{l}$ for $%
\mathbf{i}=(i_{l})_{l=1}^{d},$ $\mathbf{j}=(j_{l})_{l=1}^{d}\in \mathbb{N}%
^{d}$. It is obvious that $X^{\mathbf{p}_{1}\mathbf{,p}_{2},\ldots ,\mathbf{p%
	}_{k-1}}$ is a $\mathbb{N}^{d}$ version of $X^{p_{1},p_{2},\ldots ,p_{k-1}}$%
. Recently, Ban, Hu and Lai \cite{ban2021entropy} established the Minkowski
dimension of (\ref{3}).
Related works on the dimension theory of the multidimensional multiple sum
can also be found in \cite{brunet2020dimensions}. Let 
\[
m_{1}\geq m_{2}\geq \cdots \geq m_{d}\geq 2,\text{ }\Sigma _{m_{1},\ldots
	,m_{d}}=\left( \Sigma _{m_{1}}\times \cdots \times \Sigma _{m_{d}}\right) ^{%
	\mathbb{N}}\text{.}
\]%
Define 
\begin{equation}
X_{\Omega }^{m_{1},\ldots ,m_{d}}=\{(x_{i}^{(1)},\ldots ,x_{i}^{(d)})\in
\Sigma _{m_{1},\ldots ,m_{d}}:(x_{iq^{l}}^{(1)},\ldots ,x_{iq^{l}}^{(d)})\in
\Omega ,\forall q\nmid i\}.  \label{11}
\end{equation}%
The set (\ref{11}) is called \emph{self-affine sponges under the action of
	multiplicative integers}. Brunet studies the dimensions of (\ref{11}) and
establish the associated Ledrappier-Young formula.

It is stressed that the problems of multifractal analysis and dimension
formula of multiple average on `multidimensional\textbf{\ }lattices' are new
and difficult. The difficulty is that it is not easy to decompose the
multidimensional lattices into independent sublattices according to the
given `multiple constraints', e.g., the $\mathbf{p}_{i}^{\prime }$s in (\ref%
{3}), and calculate its density among the entire lattice. Fortunately, the
technique developed in \cite{ban2021entropy} is useful and leads us to
investigate the LDP for the
multidimensional multiple averages launched by Carinci et al. in 2012 \cite%
{carinci2012nonconventional}, and multidimensional weighted multiple sum
mentioned in \cite{fan2021multifractal}. Both topics are described in the
following two paragraphs.

{\bf LDP for multiple averages on $\mathbb{Z}$.} Let $\mathcal{A}=\{+1,-1\}$. Denote
by $\mathbb{P}_{r}$ the product of Bernoulli with the parameter $r$ on $\mathcal{A}$. For $\sigma \in \mathcal{A}^{\mathbb{Z}} $, the authors \cite%
{carinci2012nonconventional} study the thermodynamic limit of the \emph{free energy function} associated to the sum 
\begin{equation}
S_{N}(\sigma )=\sum_{i=1}^{N}\sigma _{i}\sigma _{2i},  \label{7}
\end{equation}
defined as 
\[
F_{r}(\beta )=\lim_{N\rightarrow \infty }\frac{1}{N}\log \mathbb{E}%
_{r}(e^{\beta S_{N}})\text{.}
\]%
We note that if we think (\ref{7}) as a Hamiltonian and the parameter $\beta 
$ as the inverse temperature in the lattice spin systems on $\mathcal{A}^{\mathbb{Z}} $, this
is the simplest version of the \emph{multiplicative Ising model} defined in
\cite{chazottes2014thermodynamic}. Note that the Hamiltonian (\ref{7}) is long-range, non-translation invariant interaction and much more difficult to treat. In 
\cite{carinci2012nonconventional}, the authors prove that the sequence of
multiple average $\frac{S_{N}}{N}$ satisfies a LDP with the rate function 
\begin{equation}
I_{r}(x)=\sup_{\beta \in \mathbb{R}}(\beta x-F_{r}(\beta ))\text{,}
\label{4}
\end{equation}%
where 
\[
F_{r}(\beta )=\log ([r(1-r)]^{\frac{3}{4}}|v^{T}\cdot e_{+}|\Lambda _{+})+%
\mathcal{G(\beta )}\text{.}
\]%
The reader is referred to \cite{carinci2012nonconventional} for the explicit
definitions of $v$, $e_{+}$, $\Lambda _{+}$ and $\mathcal{G(\beta )}$.
Roughly speaking, the LDP characterizes the limit behavior, as $\epsilon
\rightarrow 0$, of a family of probability measures $\{\mu _{\epsilon }\}$
on a probability space $(X,\mathcal{B})$ in terms of a rate function. In 
\cite{carinci2012nonconventional}, the \emph{rate function} associated with the
multiple average $\frac{S_{N}}{N}$ is defined by 
\begin{equation}
I_{r}(x)=\lim_{\epsilon \rightarrow 0}\lim_{N\rightarrow \infty }-\frac{1}{N}%
\log \mathbb{P}_{r}\left( \frac{S_{N}}{N}\in \lbrack x-\epsilon ,x+\epsilon
]\right) \text{.}  \label{5}
\end{equation}%
The authors prove that (\ref{5}) exists and satisfies the \emph{%
	Fenchel-Legendre transform} (\ref{4}) of the free energy function $%
F_{r}(\beta )$. If $F_{r}(\beta )$ is differentiable, say $F_{r}^{\prime
}(\eta )=y$, the rate function can be clearly demonstrated (Lemma 2.2.31 
\cite{dembo2009ldp}). Namely,  
\[
I_{r}(y)=\eta y-F_{r}(\eta )\text{.}
\]%
Thus, the characterization the differentiability of the free energy function $F_r(\beta)$ is also a major subject of the LDP, and it highly related to the phase
transition phenomena of the multiplicative Ising model (cf. \cite{georgii2011gibbs}). We
refer the reader to \cite{dembo2009ldp, ellis2006entropy} for the formal definitions of LDP and Fenchel-Legendre transform. The multiplicative Ising
model with boundary conditions is also considered in \cite%
{chazottes2014thermodynamic}. The first part of this investigation is to
extend the work of \cite{carinci2012nonconventional} and \cite%
{chazottes2014thermodynamic} to $\mathbb{N}^{d}$ without and with boundary
conditions in Section 3 and Section 5 respectively. We also extend some
results of weighted multiple average \cite{fan2021multifractal} to $\mathbb{N}^d$ version, and
describe them below.

{\bf Multifractal analysis for weighted sums on $\mathbb{Z}$.} Let $(X,T)$ be a topological dynamical system. Fan \cite{fan2021multifractal}
studies the multifractal analysis of the \emph{weighted (Birkhoff) sum} as
follows.%
\begin{equation}
S_{N}^{(w)}f(x)=\sum_{n=1}^{N}w_{n}f(T^{n}x).  \label{6}
\end{equation}%
Suppose $(w_{n})$ takes a finite number of values and $%
f_{n}(x)=x_{n}g_{n}(x_{n-1},\ldots )$, where $g_{n}$ depends on finite
number of coordinates ($(\rm C1)$, Theorem \ref{theorem:4.1-1}) and $\left( w_{n}\right) $
satisfies the frequence condition ($(\rm C2)$, Theorem \ref{theorem:4.1-1}). The \emph{spectrum of the
Hausdorff dimension} of the level set $E(\alpha )$(defined in (\ref{level set E})) is obtained in Theorem \ref{theorem:4.1-1}.
Let $\mu $ be the M\"{o}bius function, the author also considers the level
set $F(\alpha )$ (defined in (\ref{level F})). The dimension spectrum for the level set $%
F(\alpha )$ is also obtained in Theorem \ref{theorem:4.1-2}. The second part of this study is
to establish the LDP based on the weighted multiple sum 
\begin{equation}
S_{N}^{(w)}=\sum_{i=1}^{N}w_{n}\sigma _{i}\sigma _{2i}  \label{13}
\end{equation}
in $\mathbb{N}^{d}$. Our main results are presented below.       

Suppose $\mathbf{N}=(N_{1},N_{2},\ldots N_{d})\in \mathbb{N}^{d}$ and $%
\sigma \in \mathcal{A}^{\mathbb{N}^{d}}$, the (\emph{multidimensional}) 
\emph{multiple sum} is defined as 
\begin{equation}
S_{N_{1}\times N_{2}\times \cdots \times N_{d}}^{\mathbf{p}}(\sigma
)=\sum_{i_{1}=1}^{N_{1}}\sum_{i_{2}=1}^{N_{2}}\cdots
\sum_{i_{d}=1}^{N_{d}}\sigma _{\mathbf{i}}\sigma _{\mathbf{p\cdot i}}\text{.}
\label{14}
\end{equation}%
Following \cite{carinci2012nonconventional}, let $\mathbb{P}_{r}$ be a
product of Bernoulli with the parameter $r$ over two symbols on $\mathcal{A}$%
. The \emph{free energy function} associated with the sum $S_{N_{1}\times
	N_{2}\times \cdots \times N_{d}}^{\mathbf{p}}$ is set as 
\begin{equation}
F_{r}(\beta )=\lim_{\mathbf{N\rightarrow \infty }}\frac{1}{N_{1}N_{2} \cdots N_{d}}\log \mathbb{E}_{r}(\exp (\beta
S_{N_{1}\times N_{2}\times \cdots \times N_{d}}^{\mathbf{p}}))\text{.}\footnote{To shorten notation, we write $\mathbf{N}\rightarrow \infty$ instead of $N_1,N_2,...,N_d\rightarrow \infty$.}
\label{9}
\end{equation}%

 The associated \emph{large deviation rate function} of the multiple average 
\begin{equation} \label{multiple sum average}
 \frac{S_{N_{1}\times N_{2}\times \cdots \times N_{d}}^{\mathbf{p}}}{%
 	N_{1}N_{2}\cdots N_{d}}
\end{equation} 
is defined as  
\begin{equation}
I_{r}(x)=\lim_{\varepsilon \rightarrow 0}\lim_{\mathbf{N}\rightarrow \infty
}-\frac{1}{N_{1}N_{2}\cdots N_{d}}\log \mathbb{P}%
_{r}\left( \frac{S_{N_{1}\times N_{2}\times \cdots \times N_{d}}^{\mathbf{p}}%
}{N_{1}N_{2} \cdots  N_{d}}\in \lbrack x-\varepsilon
,x+\varepsilon ]\right) \text{.}  \label{8}
\end{equation}%
In Theorem \ref{theorem:3.1}, the explicit formula of $F_{r}(\beta )$ is derived and $%
\beta \mapsto F_{r}(\beta )$ is differentiable. Furthermore, the
multiple average (\ref{multiple sum average}) satisfies
the LDP. Due to the fact that $\beta \rightarrow F_{r}(\beta )$ is
differentiable, the explicit formula of $I_{r}(x)$ is obtained.
Surprisingly, the formula of $F_{r}(\beta )$ indicates that $I_{1/2}(x)$ is
independent of the dimension $d\in \mathbb{N}$ and $\mathbf{p}\in \mathbb{N}%
^{d}$. On the other hand, let $\mathbf{w}=(w_{\mathbf{i}})_{\mathbf{i}\in 
	\mathbb{N}^{d}}$, the \emph{weighted multiple sum} is defined
as 
\begin{equation}\label{weighted multiple sum}
S_{N_{1}\times N_{2}\times \cdots \times N_{d}}^{\mathbf{p,w}%
}=\sum_{i_{1}=1}^{N_{1}}\sum_{i_{2}=1}^{N_{2}}\cdots
\sum_{i_{d}=1}^{N_{d}}w_{\mathbf{i}}\sigma _{\mathbf{i}}\sigma _{\mathbf{%
		p\cdot i}}\text{.}
\end{equation}

Denote by $F_{r}^{\mathbf{w}}(\beta )$ (resp. $I_{r}^{\mathbf{w}}(x)$) the
corresponding free energy function (resp. large deviation rate function) of the
weighted multiple average 
\begin{equation}
\frac{S_{N_{1}\times N_{2}\times \cdots \times N_{d}}^{\mathbf{p},\mathbf{w}}%
}{N_{1}N_{2} \cdots  N_{d}}  \label{12}
\end{equation}
as (\ref{9}) (resp. (\ref{8})). The formula of $F_{1/2}^{\mathbf{w}}(\beta )$
is rigorously calculated in Theorem \ref{theorem:4.1}, the LDP for
the average (\ref{12}) are also given therein. It is worthy emphasize that
the formula $I_{1/2}^{\mathbf{w}}$ in Theorem \ref{theorem:4.1} is almost identical to the
dimension formula (\ref{level set E}) established in Theorem \ref{theorem:4.1-1} \cite%
{fan2021multifractal}, and it also does not depend on the dimension $d\in 
\mathbb{N}$ and the multiple constraint $\mathbf{p}\in \mathbb{N}^{d}$. In
addition, similar results are also obtained if $(w_{\mathbf{i}})_{\mathbf{i}%
	\in \mathbb{N}^{d}}$ is the M\"{o}bius function (Corollary \ref{coro mobi}). Finally, the
boundary conditions on the multiple sum (\ref{14}) are imposed and the
corresponding energy functions are defined. The explicit formulae of these
energy functions are determined in Section 5.

\section{Preliminaries}
In this section, we provide necessary materials and results on the decomposition of the multidimensional lattice $\mathbb{N}^d$ into independent sublatticies and calculate their densities. 

Given $p_1,p_2,...,p_d\geq2$ and $N_1,N_2,...,N_d\geq 1$, let $\mathcal{M}_{{\bf p}}=\{ (p_1^m,p_2^m,...,p_d^m) :m\geq 0\}$ be the subset of $\mathbb{N}^d$, denoting $\mathcal{M}_{{\bf p}}({\bf i})=\{ (i_1p_1^m,i_2p_2^m,...,i_dp_d^m) :m\geq 0\}$ as the lattice $\mathcal{M}_{\bf p}$ starts at ${\bf i}$, and let $\mathcal{I}_{{\bf p}}=\{ {\bf i} : p_j\nmid i_j\mbox{ for some }1\leq j \leq d \}$ be the complementary index set of $\mathbb{N}^d$.

More definitions are needed to characterize the partition of $N_1\times N_2\times \cdots \times N_d$ lattice. Let $\mathcal{N}_{N_1\times N_2\times \cdots \times N_d}=\{ {\bf i} : 1\leq i_j\leq N_j \mbox{ for all } 1\leq j \leq d  \}$ be $N_1\times N_2\times \cdots \times N_d$ lattice and $\mathcal{L}_{N_1\times N_2\times \cdots \times N_d}({\bf i})= \mathcal{M}_{{\bf p}}({\bf i})\cap \mathcal{N}_{N_1\times N_2\times \cdots \times N_d}$ be the subset of $\mathcal{M}_{\bf p}({\bf i}) $ in $N_1\times N_2\times \cdots \times N_d$ lattice. Then we define $\mathcal{J}_{N_1\times N_2\times \cdots \times N_d;\ell}= \{ {\bf i}\in \mathcal{N}_{N_1\times N_2\times \cdots \times N_d} : |\mathcal{L}_{N_1\times N_2\times \cdots \times N_d}({\bf i})|=\ell  \}$ as the set of points such that $\mathcal{M}_{\bf p}({\bf i})$ in $N_1\times N_2\times \cdots \times N_d$ lattice with length exactly $\ell$, where $| \cdot |$ denotes cardinal numbers. Let $\mathcal{K}_{N_1\times N_2\times \cdots \times N_d;\ell}=\{ {\bf i}\in \mathcal{I}_{{\bf p}}\cap \mathcal{N}_{N_1\times N_2\times \cdots \times N_d} : |\mathcal{L}_{N_1\times N_2\times \cdots \times N_d}({\bf i})|=\ell  \}$ be the set of points in $\mathcal{I}_{\bf p}$ that $\mathcal{M}_{\bf p}({\bf i})$ in $N_1\times N_2\times \cdots \times N_d$ lattice with length exactly $\ell$. The following lemma computes the limit of density of $\mathcal{K}_{N_1\times N_2\times \cdots \times N_d;\ell}$ which is $\mathbb{N}^d$ version of Lemmas 2.1 and 2.2 in \cite{ban2021entropy}.

\begin{lemma}\label{lemma:2.2}
	For $p_1,p_2,...,p_d\geq2$,
	\begin{equation*}  
		\mathbb{N}^d=\displaystyle\mathop{\coprod}\limits_{{\bf i} \in \mathcal{I}_{{\bf p}}}\mathcal{M}_{{\bf p}}({\bf i}).
	\end{equation*}  
\end{lemma}

\begin{proof}
	We first claim that for all $ {\bf i} \neq {\bf i}' \in  \mathcal{I}_{\bf p}$, $\mathcal{M}_{\bf p}({\bf i})\cap\mathcal{M}_{\bf p}({\bf i}')=\emptyset$. Indeed, suppose this does not hold. Then there exist ${\bf i} \neq {\bf i}' \in \mathcal{I}_{\bf p}$ such that $\mathcal{M}_{\bf p}({\bf i})\cap\mathcal{M}_{\bf p}({\bf i}')\neq\emptyset$. Since ${\bf i} \neq {\bf i}'$, then there exist $ m_1\neq m_2\geq 0$ such that $ (i_1p_1^{m_1},..., i_dp_d^{m_1})=(i_1'p_1^{m_2},..., i_d'p_d^{m_2})$. Without loss of generality, we assume $m_1>m_2$, then  $i_kp_k^{m_1-m_2} = i_k'$ for all $1\leq k \leq d$, which gives $ p_k | i_k'$ for all $1\leq k \leq d$. This contradicts ${\bf i}' \in \mathcal{I}_{\bf p}$. 
	It remains to show that the equality holds. For ${\bf i} \in \mathbb{N}^d$, then $i_k=i_k'p_k^{\alpha_k}$ with $p_k\nmid i_k'$ and $\alpha_k\geq 0 $ for all $1\leq k \leq d$. Take $\gamma=\min_k \{\alpha_k\}$, then $(\frac{i_1}{p_1^{\gamma}},...,\frac{i_d}{p_d^{\gamma}}) \in \mathcal{I}_{\bf p} $, which implies ${\bf i} \in \mathcal{M}_{\bf p}(\frac{i_1}{p_1^{\gamma}},...,\frac{i_d}{p_d^{\gamma}})$. Since the converse is clear, the proof is thus completed.
\end{proof}

\begin{lemma}\label{lemma:1.5} 
	For $N_1, N_2,...,N_d$, and $\ell \geq 1$, we have the following assertions.	
\item[1.]$\displaystyle	|\mathcal{J}_{N_1\times N_2\times \cdots \times N_d;\ell}| = \prod_{k=1}^{d}\left\lfloor\frac{N_k}{p_k^{\ell-1}}\right\rfloor-\prod_{k=1}^{d}\left\lfloor\frac{N_k}{p_k^\ell}\right\rfloor.$
\item[2.]$\displaystyle\lim_{{\bf N}\to \infty}\frac{| \mathcal{K}_{N_1\times N_2\times \cdots \times N_d;\ell} | }{ | \mathcal{J}_{N_1\times N_2\times \cdots \times N_d;\ell} | }=1-\frac{1}{p_1p_2\cdots p_d}.$
\item[3.] $\displaystyle\lim_{{\bf N}\to \infty}\frac{1}{N_1\cdots N_d}\sum_{\ell=1}^{N_1\cdots N_d}|\mathcal{K}_{N_1 \times \cdots \times N_d;\ell}|\log F_{\ell}
=\sum_{\ell=1}^{\infty}\displaystyle{\lim_{{\bf N}\to \infty}}\frac{|\mathcal{K}_{N_1\times \cdots \times N_d;\ell}|}{N_1\cdots N_d}\log F_{\ell}.$		
\end{lemma}
\begin{proof}	
	\item[\bf 1.] Since $|\mathcal{L}_{N_1\times \cdots \times N_d}({\bf i})|=\ell$, we have
	\[
	\mathcal{J}_{N_1\times \cdots \times N_d;\ell}=\{ {\bf i} : i_kp_k^{\ell-1}\leq N_k \mbox{ for all }1\leq k \leq d \}
	\cap \left( \cup_{k=1}^{d} \{{\bf i} : i_kp_k^\ell>N_k \} \right).
	\]
	  Thus, the Inclusion-exclusion principle infers that
	\begin{equation*}
		\begin{aligned}
			|\mathcal{J}_{N_1\times \cdots \times N_d;\ell}|
			&=\left| \bigcup_{n=1}^{d}\left( A
			\cap \{{\bf i} : i_np_n^\ell>N_n \} \right)\right|\\
			&=\sum_{n=1}^{d}\left|A
			\cap  \{{\bf i} : i_np_n^\ell>N_n \} \right|\\
			&-\sum_{1\leq n_1<n_2 \leq d}\left|A
			\cap  \{{\bf i} : i_{n_1}p_{n_1}^\ell>N_{n_1} \mbox{ and }i_{n_2}p_{n_2}^\ell>N_{n_2} \}\right|	\\	
			&+\sum_{1\leq n_1<n_2<n_3\leq d}\left|A
			\cap  \{{\bf i} : i_{n_1}p_{n_1}^\ell>N_{n_1} ,i_{n_2}p_{n_2}^\ell>N_{n_2}\mbox{ and }i_{n_3}p_{n_3}^\ell>N_{n_3} \}\right|\\
			&-\cdots+(-1)^{d-1}\left|A
			\cap  \{{\bf i} : i_{1}p_{1}^\ell>N_{1} ,i_{2}p_{2}^\ell>N_{2},...,i_{d}p_{d}^\ell>N_{d} \}\right|,
				\end{aligned}
	\end{equation*}	
where $A=\{ {\bf i} : i_kp_k^{\ell-1}\leq N_k \mbox{ for all }1\leq k \leq d \}$.

It follows that
			\begin{equation*}
			\begin{aligned}	
			|\mathcal{J}_{N_1\times \cdots \times N_d;\ell}|&=\sum_{n=1}^{d}\left[ \left(\left\lfloor\frac{N_n}{p_n^{\ell-1}}\right\rfloor-\left\lfloor\frac{N_n}{p_n^\ell}\right\rfloor\right)\prod_{k\neq n}\left\lfloor\frac{N_k}{p_k^{\ell-1}}\right\rfloor \right]\\
			&-\sum_{1\leq n_1<n_2 \leq d}\left[ \prod_{k_1=n_1,n_2}\left(\left\lfloor\frac{N_{k_1}}{p_{k_1}^{\ell-1}}\right\rfloor-\left\lfloor\frac{N_{k_1}}{p_{k_1}^\ell}\right\rfloor\right)\prod_{k_2\neq n_1,n_2}\left\lfloor\frac{N_{k_2}}{p_{k_2}^{\ell-1}}\right\rfloor \right]\\
			&+\sum_{1\leq n_1<n_2<n_3\leq d}\left[ \prod_{k_1=n_1,n_2,n_3}\left(\left\lfloor\frac{N_{k_1}}{p_{k_1}^{\ell-1}}\right\rfloor-\left\lfloor\frac{N_{k_1}}{p_{k_1}^\ell}\right\rfloor\right)\prod_{k_2\neq n_1,n_2,n_3}\left\lfloor\frac{N_{k_2}}{p_{k_2}^{\ell-1}}\right\rfloor \right]\\
			&-\cdots+(-1)^{d-1}\prod_{k=1}^{d}\left(\left\lfloor\frac{N_{k}}{p_{k}^{\ell-1}}\right\rfloor-\left\lfloor\frac{N_{k}}{p_{k}^\ell}\right\rfloor\right).
					\end{aligned}
		\end{equation*}	
Thus, we have  			
					\begin{equation*}
			\begin{aligned}		
				|\mathcal{J}_{N_1\times \cdots \times N_d;\ell}|&=\prod_{k=1}^{d}\left\lfloor\frac{N_k}{p_k^{\ell-1}}\right\rfloor-\prod_{k=1}^{d}\left\lfloor\frac{N_k}{p_k^\ell}\right\rfloor.
		\end{aligned}
	\end{equation*}	
	\item[\bf 2.] For $m_{2}^{(i)}>m_{1}^{(i)}\geq 1$, $1\leq i \leq d$ and let the rectangular lattice
	\begin{equation*}
		\mathcal{R}_{m_{1}^{(1)},m_{2}^{(1)};...;m_{1}^{(d)},m_{2}^{(d)}}=\left\{{\bf i}\in\mathbb{N}^{d}:m_{1}^{(k)}\leq i_k\leq m_{2}^{(k)} \text{ for all } 1\leq k \leq d\right\}.
	\end{equation*}
	Clearly, the complement of $\mathcal{I}_{\bf p}$ is $\mathcal{I}^{c}_{\bf p}=\{{\bf i}:p_k\mid i_k \text{ for all }1\leq k \leq d\}$ and
	\begin{equation*}
		\left|	\mathcal{R}_{m_{1}^{(1)},m_{2}^{(1)};...;m_{1}^{(d)},m_{2}^{(d)}}\cap\mathcal{I}_{\bf p}\right|=\left|	\mathcal{R}_{m_{1}^{(1)},m_{2}^{(1)};...;m_{1}^{(d)},m_{2}^{(d)}} \right|-
		\left|	\mathcal{R}_{m_{1}^{(1)},m_{2}^{(1)};...;m_{1}^{(d)},m_{2}^{(d)}}\cap\mathcal{I}^{c}_{\bf p} \right|.
	\end{equation*}
	Thus,
	\begin{equation*}
		\begin{aligned}
			&\left|	\mathcal{R}_{m_{1}^{(1)},m_{2}^{(1)};...;m_{1}^{(d)},m_{2}^{(d)}}\cap\mathcal{I}_{\bf p} \right|\\
			&\geq \left|	\mathcal{R}_{m_{1}^{(1)},m_{2}^{(1)};...;m_{1}^{(d)},m_{2}^{(d)}} \right|-\frac{1}{p_1p_2\cdots p_d}\left|\mathcal{R}_{m_{1}^{(1)},m_{2}^{(1)}+2p_1;...;m_{1}^{(d)},m_{2}^{(d)}+2p_d} \right|
		\end{aligned}
	\end{equation*}
	and
	\begin{equation*}
		\begin{aligned}
			&\left|	\mathcal{R}_{m_{1}^{(1)},m_{2}^{(1)};...;m_{1}^{(d)},m_{2}^{(d)}}\cap\mathcal{I}_{\bf p} \right|\\
			&\leq \left|	\mathcal{R}_{m_{1}^{(1)},m_{2}^{(1)};...;m_{1}^{(d)},m_{2}^{(d)}} \right|-\frac{1}{p_1p_2\cdots p_d} \left|\mathcal{R}_{m_{1}^{(1)},m_{2}^{(1)}-2p_1;...;m_{1}^{(d)},m_{2}^{(d)}-2p_d} \right|.
		\end{aligned}
	\end{equation*}
	Then, by the Squeeze theorem,
	\begin{equation*}
		\underset{1\leq k \leq d}{\underset{m_{2}^{(k)}-m_{1}^{(k)}\rightarrow\infty }{\lim}}\frac{|	\mathcal{R}_{m_{1}^{(1)},m_{2}^{(1)};...;m_{1}^{(d)},m_{2}^{(d)}}\cap\mathcal{I}_{\bf p}|}{|	\mathcal{R}_{m_{1}^{(1)},m_{2}^{(1)};...;m_{1}^{(d)},m_{2}^{(d)}}|}=1-\frac{1}{p_1p_2\cdots p_d}.
	\end{equation*}
	Consequently,
	\begin{equation*}
		\begin{aligned}
			\displaystyle{\lim_{{\bf N}\to \infty}} \frac{ | \mathcal{K}_{N_1\times \cdots \times N_d;\ell} |}{|\mathcal{J}_{N_1\times \cdots \times N_d;\ell} | }  &=\displaystyle{\lim_{{\bf N}\to \infty}} \frac{ | \mathcal{J}_{N_1\times \cdots \times N_d;\ell} \cap \mathcal{I}_{\bf p}|}{|\mathcal{J}_{N_1\times \cdots \times N_d;\ell} | }=1-\frac{1}{p_1p_2\cdots p_d}.
		\end{aligned}
	\end{equation*}	
	\item[\bf 3.] Define
	\begin{center}
		$\bar{K}_{N_1\times \cdots \times N_d;\ell}=
		\left\{
		\begin{array}{ll}
		|\mathcal{K}_{N_1\times \cdots \times N_d;\ell} |& \mbox{ if } \ell \leq N_1\cdots N_d,
		\\
		0 & \mbox{ if } \ell > N_1\cdots N_d.
		\end{array}
		\right.$
	\end{center}
	Then
	\begin{equation*}  
		\begin{aligned}
			\displaystyle{\lim_{{\bf N}\to \infty}}\frac{1}{N_1\cdots N_d}\sum_{\ell=1}^{N_1\cdots N_d} | \mathcal{K}_{N_1\times \cdots \times N_d;\ell} | \log F_{\ell}=\displaystyle{\lim_{{\bf N}\to \infty}}\frac{1}{N_1\cdots N_d}\sum_{\ell=1}^{\infty}\bar{K}_{N_1\times \cdots \times N_d;\ell}\log F_{\ell}.
		\end{aligned}
	\end{equation*}  
	Hence from the Weierstrass M-test with
	\begin{equation*}  
		\begin{aligned}
			\displaystyle\frac{1}{N_1\cdots N_d}\left|\bar{K}_{N_1\times \cdots \times N_d;\ell}\log F_{\ell}\right|
			&\leq  \frac{1}{N_1\cdots N_d} |\mathcal{J}_{N_1\times \cdots \times N_d;\ell} | \log F_{\ell}\\
			&=  \frac{1}{N_1\cdots N_d}\left(\prod_{k=1}^{d}\left\lfloor\frac{N_k}{p_k^{\ell-1}}\right\rfloor-\prod_{k=1}^{d}\left\lfloor\frac{N_k}{p_k^\ell}\right\rfloor\right)\log F_{\ell} \\
			&\leq  \frac{1}{N_1\cdots N_d}\left(\frac{N_1\cdots N_d}{p_1^{\ell-1}p_2^{\ell-1}\cdots p_d^{\ell-1}}\right) \log F_{\ell}\\
			&= \frac{1}{(p_1p_2\cdots p_d)^{\ell-1}}\log F_{\ell}
		\end{aligned}
	\end{equation*}  
	for all  $N_1,...,N_d\in \mathbb{N}$ and $\displaystyle\sum_{\ell=1}^{\infty}\frac{\log F_{\ell}}{(p_1p_2\cdots p_d)^{\ell-1}}<\infty$, we deduce that $\displaystyle\sum_{\ell=1}^{\infty}\frac{\bar{K}_{N_1\times \cdots \times N_d;\ell}\log F_{\ell}}{N_1\cdots N_d}$ converges uniformly in $N_1,...,N_d$. 
	
	This implies
\begin{equation*}
\begin{aligned}
			\displaystyle{\lim_{{\bf N}\to \infty}}\frac{1}{N_1\cdots N_d}\sum_{\ell=1}^{N_1\cdots N_d} | \mathcal{K}_{N_1\times \cdots \times N_d;\ell} | \log F_{\ell}&=\displaystyle{\lim_{{\bf N}\to \infty}}\frac{1}{N_1\cdots N_d}\sum_{\ell=1}^{\infty}\bar{K}_{N_1\times \cdots \times N_d;\ell}\log F_{\ell}\\
			&=\displaystyle\sum_{\ell=1}^{\infty}{\lim_{{\bf N}\to \infty}}\frac{1}{N_1\cdots N_d}\bar{K}_{N_1\times \cdots \times N_d;\ell}\log F_{\ell}\\
			&=\displaystyle\sum_{\ell=1}^{\infty}{\lim_{{\bf N}\to \infty}}\frac{1}{N_1\cdots N_d} | \mathcal{K}_{N_1\times \cdots \times N_d;\ell} | \log F_{\ell}.
\end{aligned}
\end{equation*}
	The proof is complete.
\end{proof}

\section{LDP of multiple averages on $\mathbb{N}^d$}

In this section, we establish the LDP for the mutiple average (\ref{multiple sum average}), where the multiple sum $S^{{\bf p}}_{N_1\times N_2\times \cdots \times N_d}$ is defined in (\ref{14}). The associated free energy function $F_r(\beta)$ and the large deviation rate function $I_r(x)$ are also defined in (\ref{9}) and (\ref{8}) respectively.

Let $N_1,N_2,...,N_d\in \mathbb{N}$ and ${\bf p}=(p_1,p_2,...,p_d)\in \mathbb{N}^d$ with $\gcd(p_i,p_j)=1$ for all $1\leq i<j\leq d$. The explicit formula of $F_r(\beta)$ and LDP of the multiple average (\ref{multiple sum average}) are established in Theorem \ref{theorem:3.1} below. The following theorem is essential in the proofs of our results.

\begin{theorem}[G\"{a}rtner-Ellis \cite{dembo2009ldp}, Theorem 2.3.6]\label{theorem:GE}
	If the limit $(\ref{9})$ exists and such limit is finite in a neighborhood of origin, then $(\ref{8})$ is equal to the Fenchel-Legendre transform of $(\ref{9})$, i.e.,
		\begin{equation*}
	\begin{aligned}
	I_r(x)=\sup_{\beta \in \mathbb{R}} \left( \beta x -F_r(\beta) \right).
	\end{aligned}
	\end{equation*}	
Moreover, if the function $(\ref{9})$ is differentiable, say $(F_r)'(\eta)=y$, then 
	\begin{equation*}
	\begin{aligned}
	I_r(y)=\eta y -F_r(\eta).
\end{aligned}
\end{equation*}	
\end{theorem}

\begin{theorem}\label{theorem:3.1}
	For any $d\geq 1$ and $p_1,p_2,...,p_d\geq 1$ with $\gcd(p_i,p_j)=1$ for $1\leq i< j\leq d$. Then
	\item[1.]The explicit formula of the free energy function associated to multiple sum $S_{N_1\times N_2\times \cdots \times N_d}^{{\bf p}}$ is
		\begin{equation}\label{F_r(beta)}
		\begin{aligned}
			F_r(\beta)
			=\frac{2p_1p_2\cdots p_d-1}{2p_1p_2\cdots p_d}\log (r(1-r))+\frac{p_1p_2\cdots p_d-1}{p_1p_2\cdots p_d}\log|v^T\cdot e_+|^2 +\log\Lambda_++\mathcal{G}(\beta),
		\end{aligned}
	\end{equation}
	where $\Lambda_{\pm}$, $v^T$, $h$, $e_+$ and $\mathcal{G}(\beta)$ defined in {\rm (\ref{lambda+-}), (\ref{vT}), (\ref{h}), (\ref{e+})} and {\rm(\ref{Gbeta})} respectively.
	\item[2.]The function $F_r(\beta)$ is differentiable with respect to $\beta \in \mathbb{R}$.\\
	\item[3.]The multiple average {\rm (\ref{multiple sum average})} satisfies a LDP with the rate function
	\begin{equation*}
	\begin{aligned}
		I_{r}(x)=&\sup_{\beta \in \mathbb{R}}\left( \beta x-F_r(\beta) \right).
	\end{aligned}
\end{equation*}
Furthermore, if $(F_r)'(\eta)=y$, then $I_{r}(y)=\eta y- F_r(\eta)$. 
\end{theorem}

\begin{proof}
	\item[\bf 1.]By Lemma \ref{lemma:2.2}, we decompose the sum (\ref{14}) as
	\begin{equation}\label{106}
		\begin{aligned}
			S_{N_1\times N_2\times \cdots \times N_d}^{{\bf p}}
			&=\sum_{{\bf i}\in \mathcal{I}_{{\bf p}}}\left(\sum_{{\bf x}\in \mathcal{L}_{N_1\times N_2\times \cdots \times N_d}({\bf i})}\sigma_{{\bf x}}\sigma_{{\bf p}\cdot {\bf x}}    \right)     .                   
		\end{aligned}
	\end{equation}
	For a given ${\bf i}\in \mathcal{I}_{{\bf p}}$, each subsystem $\mathcal{L}_{N_1\times N_2\times \cdots \times N_d}({\bf i})$ is nothing else than the Hamiltonian of a one-dimensional nearest-neighbors Ising model, since
	\begin{equation*}
		\begin{aligned}
			\left\{\sum_{{\bf x}\in \mathcal{L}_{N_1\times N_2\times \cdots \times N_d}({\bf i})}\sigma_{{\bf x}}\sigma_{{\bf p}\cdot {\bf x}} \right\} \overset{\mathcal{D}}{=}\left\{   \sum_{\ell=1}^{|\mathcal{L}_{N_1\times N_2\times \cdots \times N_d}({\bf i})|}   \tau^{({\bf i})}_\ell \tau^{({\bf i})}_{\ell+1} \right\},
		\end{aligned}
	\end{equation*}
where $\tau^{({\bf i})}_\ell$ are Bernoulli with parameter $r$, independent for different values of ${\bf i}$ and for different values of $\ell$ and $\overset{\mathcal{D}}{=}$ denotes equality in distribution. We introduce the notation $Z(\beta,h,\ell+1)$ as \cite{carinci2012nonconventional} by
	\begin{equation}\label{100}
		\begin{aligned}
			Z(\beta,h,\ell+1)=\sum_{\tau \in \{-1,1\}^{\ell+1}}e^{\beta \sum_{i=1}^\ell\tau_i \tau_{i+1}+h\sum_{i=1}^{\ell+1}\tau_i}
		\end{aligned}
	\end{equation}
	for the partition function of the one-dimensional Ising model with coupling strength $\beta$ and external field $h$ in the volume $\{ 1,...,\ell \}$, with free boundary conditions. Then 
	\begin{equation}
		\begin{aligned}\label{101}
			\mathbb{E}_r\left(e^{\beta \sum_{i=1}^\ell\tau_i \tau_{i+1} }
			\right)=\left(r(1-r) \right)^{\frac{\ell+1}{2}}Z(\beta,h,\ell+1),
		\end{aligned}
	\end{equation}
where 
	\begin{equation}\label{h}
\begin{aligned}
h=\frac{1}{2}\log \left(\frac{r}{1-r} \right).
 \end{aligned}
 \end{equation}	
By the computation in (\cite{RJB-1982}, Chapter 2), (\ref{100}) becomes to
	\begin{equation}\label{102}
		\begin{aligned}
			Z(\beta,h,\ell+1)=&
			v^T\left[
			\begin{matrix}
				e^{\beta+h} & e^{-\beta}  \\
				e^{-\beta}& e^{\beta-h} 				
			\end{matrix}
			\right]^{\ell}v
			=|v^T\cdot e_+|^2 \Lambda_+^\ell+|v^T\cdot e_-|^2 \Lambda_-^\ell ,
		\end{aligned}
\end{equation}			
where 
\begin{equation}\label{vT}
\begin{aligned}	
v^T=\left(e^{\frac{h}{2}},e^{\frac{-h}{2}}\right)
\end{aligned}
\end{equation}
and
\begin{equation}\label{lambda+-}
\begin{aligned}	
\Lambda_{\pm}=e^\beta\left(\cosh(h)\pm\sqrt{\sinh^2(h)+e^{-4\beta}}\right)
\end{aligned}
\end{equation}
is the largest, resp. smallest eigenvalue of transition matrix $\left[
\begin{matrix}
e^{\beta+h} & e^{-\beta}  \\
e^{-\beta}& e^{\beta-h} 
\end{matrix}
\right]$ with $e_{\pm}$ the normalized eigenvectors corresponding to the eigenvalues $\Lambda_{\pm}$.		
			
By choosing the special normalized eigenvector $e_+$ (corr. to $\Lambda_+$),
\begin{equation}\label{e+}
\begin{aligned}	
e_+=\frac{w_+}{|| w_+ ||} \mbox{ with } 
w_+=\left( \begin{matrix}
-e^{-\beta}  \\
e^{h+\beta}-\Lambda_+
\end{matrix} \right).
\end{aligned}
\end{equation}
 Then (\ref{102}) and (\ref{e+}) give
	\begin{equation}\label{105}
	\begin{aligned}						
		Z(\beta,h,\ell+1)=&|v^T\cdot e_+|^2 \Lambda_+^\ell+(||v||^2-|v^T\cdot e_+|^2) \Lambda_-^\ell\\
		=&|v^T\cdot e_+|^2 \Lambda_+^\ell+(2\cosh(h)-|v^T\cdot e_+|^2) \Lambda_-^\ell.
	\end{aligned}
\end{equation}
Then by (\ref{106}) and (\ref{101}), (\ref{9}) becomes to
\begin{equation}\label{107}
	\begin{aligned}
		F_r(\beta)=&\lim_{{\bf N}\rightarrow\infty}\frac{1}{N_1\cdots N_d}\sum_{{\bf i}\in \mathcal{I}_{\bf p}}\log (r(1-r))^{\frac{|\mathcal{L}_{N_1\times N_2\times \cdots \times N_d}({\bf i})|+1}{2}}Z(\beta,h,|\mathcal{L}_{N_1\times N_2\times \cdots \times N_d}({\bf i})|+1).\\
	\end{aligned}
\end{equation}
Then by Lemma \ref{lemma:2.2}, Lemma \ref{lemma:1.5} and (\ref{107}), we have
	\begin{equation}\label{108}
		\begin{aligned}
		F_r(\beta)&=\sum_{\ell=1}^{\infty}\frac{(p_1p_2\cdots p_d-1)^2}{(p_1p_2\cdots p_d)^{\ell+1}}\log (r(1-r))^{\frac{\ell+1}{2}}Z(\beta,h,\ell+1).
					\end{aligned}
		\end{equation}
Combining (\ref{105}) and (\ref{108}) yields			
			\begin{equation*}
			\begin{aligned}
				F_r(\beta)&=\sum_{\ell=1}^{\infty}\frac{(p_1p_2\cdots p_d-1)^2}{(p_1p_2\cdots p_d)^{\ell+1}}\log (r(1-r))^{\frac{\ell+1}{2}}\\
			&+\sum_{\ell=1}^{\infty}\frac{(p_1p_2\cdots p_d-1)^2}{(p_1p_2\cdots p_d)^{\ell+1}}\log\left(|v^T\cdot e_+|^2 \Lambda_+^\ell+(2\cosh(h)-|v^T\cdot e_+|^2) \Lambda_-^\ell\right)\\
			&=\frac{2p_1p_2\cdots p_d-1}{2p_1p_2\cdots p_d}\log (r(1-r))\\
			&+\sum_{\ell=1}^{\infty}\frac{(p_1p_2\cdots p_d-1)^2}{(p_1p_2\cdots p_d)^{\ell+1}}\log|v^T\cdot e_+|^2 +\sum_{\ell=1}^{\infty}\frac{(p_1p_2\cdots p_d-1)^2}{(p_1p_2\cdots p_d)^{\ell+1}}\ell\log\Lambda_++\mathcal{G}(\beta)\\
			&=\frac{2p_1p_2\cdots p_d-1}{2p_1p_2\cdots p_d}\log (r(1-r))+\frac{p_1p_2\cdots p_d-1}{p_1p_2\cdots p_d}\log|v^T\cdot e_+|^2 +\log\Lambda_++\mathcal{G}(\beta),
		\end{aligned}
	\end{equation*}
where
\begin{equation}\label{Gbeta}
\begin{aligned}
\mathcal{G}(\beta)=\sum_{\ell=1}^{\infty}\frac{(p_1p_2\cdots p_d-1)^2}{(p_1p_2\cdots p_d)^{\ell+1}}\log\left(1+\left( \frac{2\cosh(h)}{|v^T\cdot e_+|^2}-1 \right)\left(\frac{\Lambda_-}{\Lambda_+} \right)^\ell\right).
\end{aligned}
\end{equation}
\item[\bf 2.]For the proof of differentiability of $\beta\mapsto F_r(\beta)$, it is enough to show the sum
		\begin{equation*}
		\begin{aligned}
			\sum_{\ell=1}^{\infty}\frac{(p_1p_2\cdots p_d-1)^2}{(p_1p_2\cdots p_d)^{\ell+1}} \left[\log\left(1+\left( \frac{2\cosh(h)}{|v^T\cdot e_+|^2}-1 \right)\left(\frac{\Lambda_-}{\Lambda_+} \right)^\ell\right)\right]'
		\end{aligned}
\end{equation*}
is converge uniformly with respect to $\beta \in \mathbb{R}$. Indeed, for all $\ell \geq 1$
			\begin{equation}\label{diff1}
		\begin{aligned}
 \left[\log\left(1+\left( \frac{2\cosh(h)}{|v^T\cdot e_+|^2}-1 \right)\left(\frac{\Lambda_-}{\Lambda_+} \right)^\ell\right)\right]'=\frac{\left[\frac{2\cosh(h)}{(v^T\cdot e_+)^2}\right]'\left(\frac{\Lambda_-}{\Lambda_+} \right)^\ell+\left[\frac{2\cosh(h)}{(v^T\cdot e_+)^2}-1\right]\ell \left(\frac{\Lambda_-}{\Lambda_+} \right)^{\ell-1}\left(\frac{\Lambda_-}{\Lambda_+} \right)'}{1+\left( \frac{2\cosh(h)}{(v^T\cdot e_+)^2}-1 \right)\left(\frac{\Lambda_-}{\Lambda_+} \right)^\ell}.
		\end{aligned}
	\end{equation}
Note that, for all $\ell \geq 1$	
\begin{equation}\label{diff2}
\begin{aligned}
0\leq\left|\frac{\Lambda_-}{\Lambda_+} \right|^\ell\leq 1.
\end{aligned}
\end{equation}	
Then (\ref{diff1}) and (\ref{diff2}) give
			\begin{equation}\label{diff3}
\begin{aligned}
\left|\left[\log\left(1+\left( \frac{2\cosh(h)}{|v^T\cdot e_+|^2}-1 \right)\left(\frac{\Lambda_-}{\Lambda_+} \right)^\ell\right)\right]'\right| \leq \left| \left[\frac{2\cosh(h)}{(v^T\cdot e_+)^2}\right]' \right|+\left[\frac{2\cosh(h)}{(v^T\cdot e_+)^2}-1\right]\ell \left| \left(\frac{\Lambda_-}{\Lambda_+} \right)' \right|.
\end{aligned}
\end{equation}
For any bounded closed interval $[a,b]\subseteq \mathbb{R}$, there exists a postive constant $M_{a,b}$ such that
			\begin{equation}\label{diff4}
\begin{aligned}
 0= \left[\frac{2\cosh(h)}{|v|^2}-1\right]\leq\left[\frac{2\cosh(h)}{(v^T\cdot e_+)^2}-1\right]\leq M_{a,b}.
\end{aligned}
\end{equation}
It remains to check that $\left| \left[\frac{2\cosh(h)}{(v^T\cdot e_+)^2}\right]' \right|$ and $\left| \left(\frac{\Lambda_-}{\Lambda_+} \right)' \right|$ are bounded on $[a,b]$.

Direct computation infers that
\begin{equation}\label{diff5}
\begin{aligned} 
\left(\frac{\Lambda_-}{\Lambda_+} \right)' =\frac{(-4e^{-4\beta})\left[ \cosh(h)-f(\beta) \right] \left[ \frac{1-e^{-4\beta}}{f(\beta)}-\cosh(h)+f(\beta)\right]}{(1-e^{-4\beta})^2},
\end{aligned}
\end{equation}
where $f(\beta)=\sqrt{\sinh^2(h)+e^{-4\beta}}$. For the boundedness of (\ref{diff5}), we apply the fact that $f(\beta)$ is bounded away from zero on the interval $[a,b]$, thus
\begin{equation}\label{diff8}
\begin{aligned} 
\left| \left(\frac{\Lambda_-}{\Lambda_+} \right)' \right|\leq K_{a,b},
\end{aligned}
\end{equation}
for some constant $K_{a,b}$. Note that the L'H\^{o}pital's rule is applied to definition of derivative and (\ref{diff8}) when $\beta=0$ and $\beta$ near 0 respectively. Namely, for $\beta=0$
\begin{equation}\label{beta=0,1}
\begin{aligned} 
\left(\frac{\Lambda_-}{\Lambda_+} \right)'(0)&=\lim_{\epsilon\rightarrow 0}\frac{\frac{\Lambda_-}{\Lambda_+}(\epsilon)-\frac{\Lambda_-}{\Lambda_+}(0)}{\epsilon}\\
&=\lim_{\epsilon\rightarrow 0}\frac{\cosh(h)-f(\beta)}{\epsilon \left( \cosh(h)+f(\beta) \right)}.
\end{aligned}
\end{equation}
Applying L'H\^{o}pital's rule to (\ref{beta=0,1}), we obtain
 \begin{equation}\label{beta=0,2}
 \begin{aligned} 
 \left| \left(\frac{\Lambda_-}{\Lambda_+} \right)'(0)\right|=\frac{1}{\cosh^2(h)}\leq 1.
 \end{aligned}
 \end{equation}
For $|\beta| \rightarrow 0^+ $, applying L'H\^{o}pital's rule twice to (\ref{diff5}), we have
 \begin{equation}\label{beta=0,3}
\begin{aligned} 
\left| \left(\frac{\Lambda_-}{\Lambda_+} \right)'(0^+)\right|=\frac{1}{\cosh^2(h)}.
\end{aligned}
\end{equation}
Then (\ref{beta=0,2}) and (\ref{beta=0,3}) imply $\left| \left(\frac{\Lambda_-}{\Lambda_+} \right)'\right|$ is a continuous function near $\beta=0$, which obtain the boundedness when $\beta$ near 0.

It is easy to check that
\begin{equation}\label{diff6}
\begin{aligned} 
e^{h-2\beta}>0 \mbox { and } \sinh(h)-\sqrt{\sinh^2(h)+e^{-4\beta}}<0,\mbox{ for all }\beta\in [a,b].
\end{aligned}
\end{equation}

Then (\ref{diff6}) gives 
\begin{equation}\label{diff9}
\begin{aligned} 
v^T\cdot e_+\neq 0, \mbox{ on }[a,b].
\end{aligned}
\end{equation}

Combining (\ref{diff9}) and the fact that $0<C<f(\beta)$ for some $C$, we have
\begin{equation}\label{diff7}
\begin{aligned} 
\left| \left[\frac{2\cosh(h)}{(v^T\cdot e_+)^2}\right]' \right|\leq N_{a,b},
\end{aligned}
\end{equation}
for some constant $N_{a,b}$.

The uniform convergence of $\beta \mapsto F_r(\beta)$ on $[a,b]$ is thus obtained by the Weierstrass M-test and (\ref{diff3}), (\ref{diff8}) and (\ref{diff7}). Since $a$ and $b$ are arbitrary, we obtain the differentiability of $F_r(\beta)$ on $\mathbb{R}$. 
	\item[\bf 3.]Theorem \ref{theorem:3.1} (3) follows from Theorem \ref{theorem:GE} and (1), (2) of Theorem \ref{theorem:3.1}.
\end{proof}

\begin{remark}
We note that when $r=\frac{1}{2}$, we have
\begin{equation*}
	\begin{aligned}
		I_{\frac{1}{2}}(y)=& \eta \frac{e^{\eta }-e^{-\eta }}{e^{\eta }+e^{-\eta }}-\log(e^{\eta }+e^{-\eta })+\log 2,
	\end{aligned}
\end{equation*}
where 
\begin{equation*}
	\begin{aligned}
		y=\frac{e^{\eta }-e^{-\eta }}{e^{\eta }+e^{-\eta }}.
	\end{aligned}
\end{equation*}
Thus, $I_{\frac{1}{2}}$ is independent on the dimension $d\in \mathbb{N}$ and the multiple constraint vector ${\bf p}=(p_1,...,p_d)\in \mathbb{N}^d$.
\end{remark}

\begin{corollary}\label{E}
For any $p_1,p_2\geq 1$ with $\gcd(p_1,p_2)=1$, 
\item[1.]The free energy function associated to sum $S_{N_1\times N_2}^{{\bf p}}$ is
\begin{equation*}
	\begin{aligned}
		F_r(\beta)=\frac{2p_1p_2-1}{2p_1p_2}\log (r(1-r))+\frac{p_1p_2-1}{p_1p_2}\log|v^T\cdot e_+|^2 +\log\Lambda_++\mathcal{G}(\beta),\\
	\end{aligned}
\end{equation*}
where 
\begin{equation*}
	\begin{aligned}
		\mathcal{G}(\beta)=\sum_{\ell=1}^{\infty}\frac{(p_1p_2-1)^2}{(p_1p_2)^{\ell+1}}\log\left(1+\left( \frac{2\cosh(h)}{|v^T\cdot e_+|^2}-1 \right)\left(\frac{\Lambda_-}{\Lambda_+} \right)^\ell\right).\\
	\end{aligned}
\end{equation*}
\item[2.]In addition, if we let ${\bf p}=(2,1)$, then the free energy function associated to sum
\begin{equation*}
\begin{aligned}
S_{N_1\times N_2}^{(2,1)}=\sum_{i_{1}=1}^{N_{1}}\sum_{i_{2}=1}^{N_{2}}\sigma _{(i_1,i_2)}\sigma _{(2i_1,i_2)}
	\end{aligned}
\end{equation*}
is 
\begin{equation}\label{Carinci 2012}
	\begin{aligned}
		F_r(\beta)=\log\left( (r(1-r))^{\frac{3}{4}} |v^T\cdot e_+| \Lambda_+\right) +\mathcal{G}(\beta),\\
	\end{aligned}
\end{equation}
where 	
\begin{equation*}
	\begin{aligned}
		\mathcal{G}(\beta)=\frac{1}{2}\sum_{\ell=1}^{\infty}\frac{1}{2^{\ell}}\log\left(1+\left( \frac{2\cosh(h)}{|v^T\cdot e_+|^2}-1 \right)\left(\frac{\Lambda_-}{\Lambda_+} \right)^\ell\right).\\
	\end{aligned}
\end{equation*}
\end{corollary}

\begin{remark}
\item[1.]Formula (\ref{Carinci 2012}) is obtained in \cite{carinci2012nonconventional} for the multiple sum (\ref{7}). Therefore, (\ref{F_r(beta)}) is a multidimensional version of the free energy function on $\mathbb{N}^d$.
\item[2.]When $r=1/2$, we have $h=0$, $\Lambda_+=e^\beta+e^{-\beta}$ and $|v^T\cdot e_+|^2=|| v ||^2=2$. This implies $\mathcal{G}(\beta)=0$ and 
\begin{equation}\label{F1/2}
	\begin{aligned}
		F_{\frac{1}{2}}(\beta)=\log \left(\frac{1}{2}(e^\beta+e^{-\beta}) \right).	
	\end{aligned}
\end{equation}
Thus, $F_{\frac{1}{2}}$ is independent on the dimension $d\in \mathbb{N}$ and the choice of ${\bf p}\in \mathbb{N}^d$.
\item[3.]When $r\neq 1/2$, we have $h=\frac{1}{2}\log(r/1-r)\neq 0$. Then for all $\beta \in \mathbb{R}$, the vector $v$ is not parallel to the vector $w_+$ and so is $e_+$, that gives 
\begin{equation*}
\begin{aligned}
|v^T\cdot e_+|^2<|| v ||^2 || e_+ ||^2 = (r(1-r))^{-\frac{1}{2}}.
	\end{aligned}
\end{equation*}
This gives the effect of multiple constraint.
\item[4.]Figures \ref{fig1} and \ref{fig2} illustrate the free energy functions for different $r\in (0,1)$ which obtained from Theorem \ref{theorem:3.1} truncating the sum to the first 100 terms. Figure \ref{fig1} is a case for $d=2$ with $p_1=2$ and $p_2=1$. In fact, it becomes to one-dimensional result in \cite{carinci2012nonconventional}. The figure \ref{fig2} is a multidimensional case for $d=5$ with $p_1=2,p_2=3,p_3=5,p_4=7$ and $p_5=11$.
\begin{figure}[h]
	\includegraphics[scale=0.45]{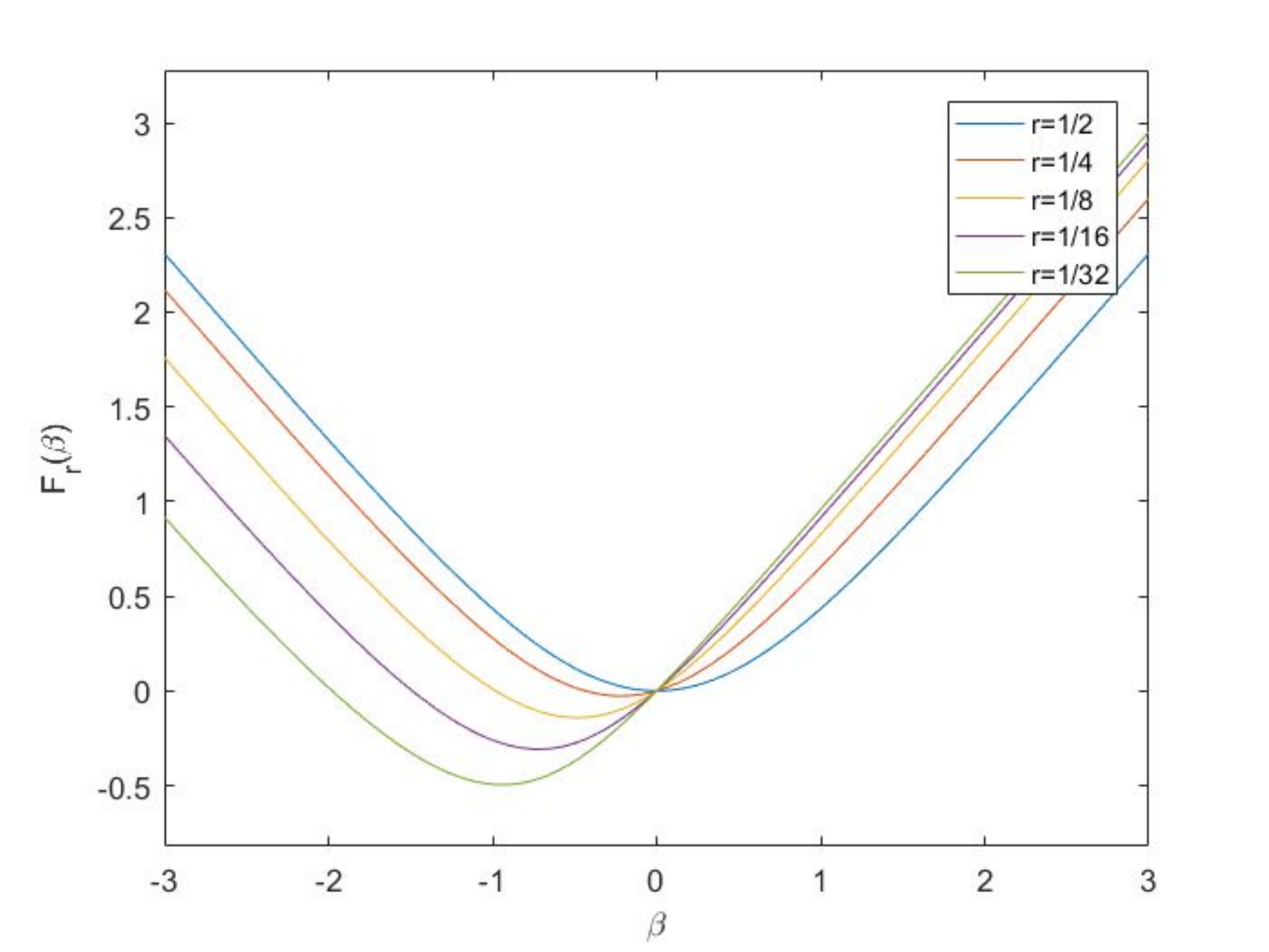}
	\caption{}
	\label{fig1}
\end{figure}
\begin{figure}[h]
		\includegraphics[scale=0.45]{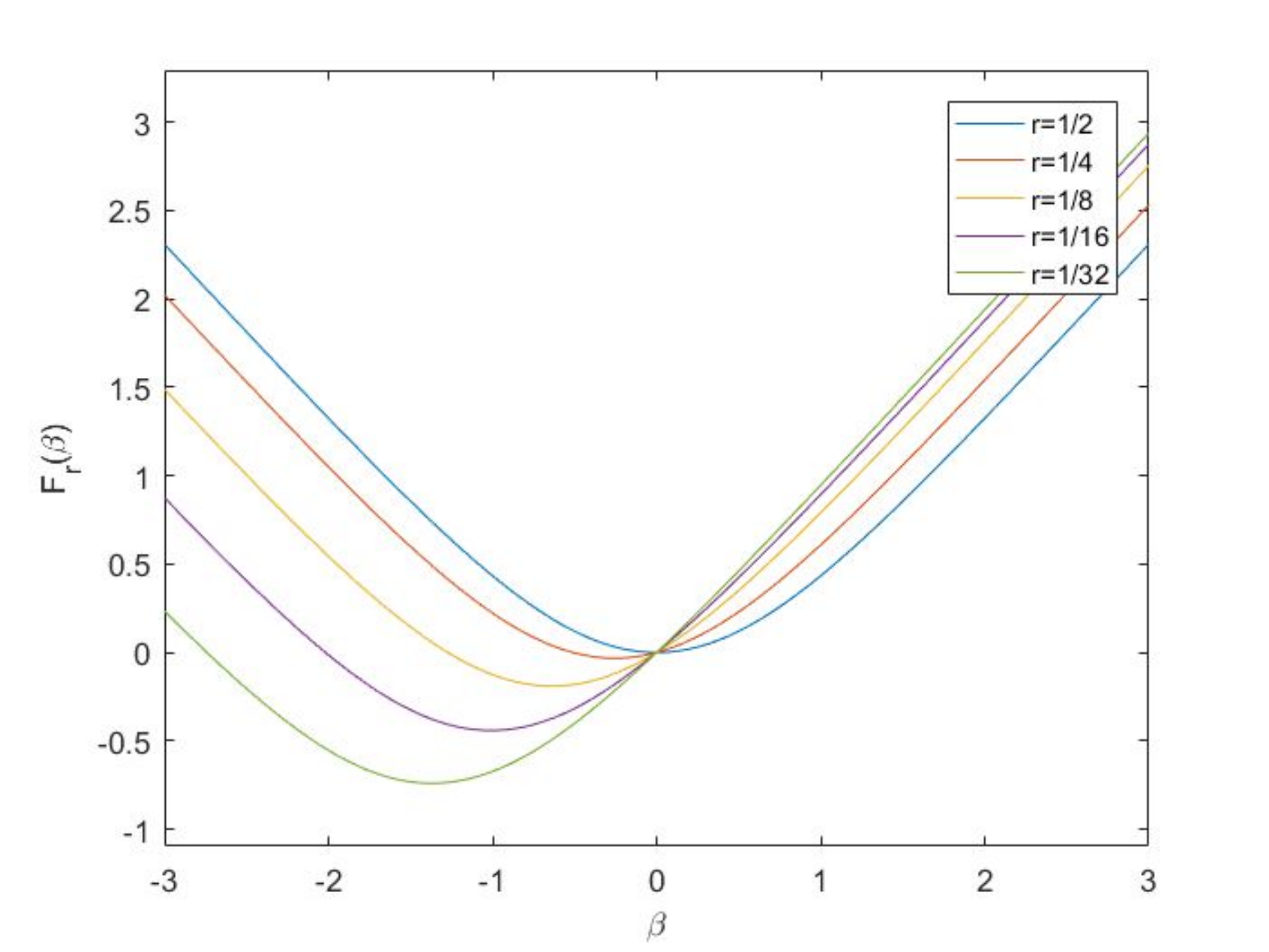}
	\caption{}
	\label{fig2}
\end{figure}
\end{remark}

\section{LDP of weighted multiple averages on $\mathbb{N}^d$}

Let $(X,T)$ be a topological dynamical system. Fan \cite{fan2021multifractal} studies the multifractal analysis of the following \emph{weighted sum}
 \begin{equation*}
 \begin{aligned}
S^{(w)}_N f(x)=\sum_{n=1}^{N} w_n f_n(T^n x)=\sum_{n=1}^{N} w_n f_n(x_n,x_{n+1},...),
	\end{aligned}
\end{equation*}
where $f=(f_n)\subseteq C(X)$. Define the level set $E(\alpha)$ according to the weighted average $\frac{S^{(w)}_N f(x)}{N}$ as
 \begin{equation}\label{level set E}
\begin{aligned}
 E(\alpha)= \left\{ x\in X : \lim_{N\rightarrow \infty} \frac{S^{(w)}_N f(x)}{N} =\alpha\right\}.
\end{aligned}
\end{equation}
The Hausdorff dimension of (\ref{level set E}) is obtained.
\begin{theorem}[A. Fan, \cite{fan2021multifractal}, Theorem 1.5]\label{theorem:4.1-1}
Let $S=\{ -1,1\}$. Suppose $(w_n)$ takes a finite number of values $v_1,v_2,...,v_m$ and $f_n(x)=x_n g(x_{n+1},x_{n+2},...)$ satisfying
\item[(C1)]$\mbox{For all } n\geq 1, ~g_n(x_{n+1},...)\in \{-1,1\} \mbox{ and only dependent on finite number of coordinates}. $
\item[(C2)]$\mbox{For all } 1\leq j\leq m, \mbox{ the following frequencies exist}$
 \begin{equation*}
\begin{aligned}
p_j:=\lim_{N\rightarrow\infty}\frac{\#\{ 1\leq n \leq N: w_n=v_j \}}{N}.
\end{aligned}
\end{equation*}
Then for $\alpha \in (-\sum p_j |v_j|,\sum p_j |v_j|)$,
 \begin{equation}\label{dimE}
\begin{aligned}
\dim E(\alpha)=\frac{1}{\log 2}\sum_{j=1}^{m}p_j \left( \log (e^{\lambda_\alpha v_j}+e^{-\lambda_\alpha v_j})-\lambda_\alpha v_j \frac{e^{\lambda_\alpha v_j}-e^{-\lambda_\alpha v_j}}{e^{\lambda_\alpha v_j}+e^{-\lambda_\alpha v_j}} \right),
\end{aligned}
\end{equation}
where $\lambda_\alpha$ is the unique solution of the equation
 \begin{equation*}
\begin{aligned}
\sum_{j=1}^{m} p_j v_j \frac{e^{\lambda_\alpha v_j}-e^{-\lambda_\alpha v_j}}{e^{\lambda_\alpha v_j}+e^{-\lambda_\alpha v_j}}=\alpha.
\end{aligned}
\end{equation*}
\end{theorem}
Let $\mu$ be the M\"{o}bius function, the author also considers the level set
 \begin{equation}\label{level F}
\begin{aligned}
F(\alpha)= \left\{ x\in \{-1,1 \}^{\mathbb{N}} : \lim_{N\rightarrow \infty} \frac{1}{N}\sum_{n=1}^{N}  \mu(n)x_n x_{n+1} =\alpha \right\}.
\end{aligned}
\end{equation}
Then the dimension spectrum for the level set $F(\alpha)$ is also obtained.

\begin{theorem}[A. Fan, \cite{fan2021multifractal}, Theorem 1.6]\label{theorem:4.1-2} For $\alpha \in (-\frac{\pi^2}{6},  \frac{\pi^2}{6})$,
	 \begin{equation*}
	\begin{aligned}
	\dim F(\alpha)=1-\frac{6}{\pi^2}+\frac{6}{\pi^2\log2}H(\frac{1}{2}+\frac{\pi^2}{12}\alpha),
	\end{aligned}
	\end{equation*}	
	where $H(x)=-x\log x-(1-x)\log(1-x)$.
\end{theorem}

In this section, we establish the LDP based on the weighted multiple average (\ref{12}). Main results of this section are presented below.

Let $N_1,N_2,...,N_d,p_1,p_2,...,p_d\geq 1$ with $\gcd(p_i,p_j)=1$ for all $1\leq i<j \leq d$. Assume the weights ${\bf w}=(w_{\bf i})_{{\bf i}\in \mathbb{N}^d}$ takes a finte number of values $v_1,v_2,...,v_m$ and the following frequencies exsist
\begin{equation}\label{111}
	\begin{aligned}
		P_k:=\lim_{{\bf N}\rightarrow\infty}\frac{\#\{{\bf x}\in \mathcal{L}_{N_1\times N_2\times \cdots \times N_d}({\bf i}):w_{\bf x}=v_k\}}{|\mathcal{L}_{N_1\times N_2\times \cdots \times N_d}({\bf i})|}
	\end{aligned}
\end{equation}
for all ${\bf i}\in\mathcal{I}_{\bf p}$ and $1\leq k\leq m$.

\begin{theorem}\label{theorem:4.1}
	For any $d\geq 1$ and $p_1,p_2,...,p_d\geq 1$ with $\gcd(p_i,p_j)=1$ for $1\leq i< j\leq d$. Assume the weighted ${\bf w}=(w_{{\bf i}})_{{\bf i}\in \mathbb{N}^d}$ satisfies frequency condition $(\ref{111})$. Then
		\item[1.]The free energy function associated to sum {\rm (\ref{weighted multiple sum})} is
	\begin{equation*}
		\begin{aligned}
		F^{\bf w}_{\frac{1}{2}}(\beta)=&\sum_{k=1}^m P_k	\log(e^{\beta v_k}+e^{-\beta v_k})-\log 2.	
		\end{aligned}
	\end{equation*}
	\item[2.]The function $F^{\bf w}_{\frac{1}{2}}(\beta)$ is differentiable with respect to $\beta \in \mathbb{R}$.
	\item[3.]The multiple average {\rm (\ref{12})} satisfies a LDP with the rate function
	\begin{equation*}
		\begin{aligned}
			I^{\bf w}_{\frac{1}{2}}(x)=&\sup_{\beta \in \mathbb{R}}\left( \beta x-\sum_{k=1}^m P_k	\log(e^{\beta v_k}+e^{-\beta v_k})+\log 2 \right).
		\end{aligned}
	\end{equation*}
Furthermore,
	\begin{equation*}
		\begin{aligned}
			I^{\bf w}_{\frac{1}{2}}(y)=&\sum_{k=1}^m P_k \left(\eta v_k\frac{e^{\eta v_k}-e^{-\eta v_k}}{e^{\eta v_k}+e^{-\eta v_k}}-\log(e^{\eta v_k}+e^{-\eta v_k})\right)+\log 2,
		\end{aligned}
	\end{equation*}
	where 
		\begin{equation*}
		\begin{aligned}
	y=\sum_{k=1}^m P_k\frac{v_k(e^{\eta v_k}-e^{-\eta v_k})}{e^{\eta v_k}+e^{-\eta v_k}}. 
			\end{aligned}
\end{equation*}
\end{theorem}

\begin{proof}
	\item[\bf 1.]Observe that the transition matrices are commute, i.e., 
	\begin{equation}\label{109}
		\begin{aligned}
			\left[
			\begin{matrix}
				e^{\beta v_i} & e^{-\beta v_i}  \\
				e^{-\beta v_i}& e^{\beta v_i} 			
			\end{matrix}
			\right]	\left[
			\begin{matrix}
				e^{\beta v_j} & e^{-\beta v_j}  \\
				e^{-\beta v_j}& e^{\beta v_j} 			
			\end{matrix}
			\right]=	\left[
			\begin{matrix}
				e^{\beta v_j} & e^{-\beta v_j}  \\
				e^{-\beta v_j}& e^{\beta v_j} 		
			\end{matrix}
			\right]	\left[
			\begin{matrix}
				e^{\beta v_i} & e^{-\beta v_i}  \\
				e^{-\beta v_i}& e^{\beta v_i} 	
			\end{matrix}
			\right],
		\end{aligned}
	\end{equation}
	for all $1\leq i<j\leq m$. According to (\ref{109}), we can sort the matrices $	\left[
	\begin{matrix}
	e^{\beta v_i} & e^{-\beta v_i}  \\
	e^{-\beta v_i}& e^{\beta v_i} 	
	\end{matrix}
	\right]$ and choose the same eigenvectors for all $1\leq i \leq m$. \\	
	Then by Lemma \ref{lemma:2.2}, Lemma \ref{lemma:1.5} and (\ref{111}), we have
	\begin{equation*}
		\begin{aligned}
			F_{\frac{1}{2}}^{\bf w}(\beta)&=\sum_{\ell=1}^{\infty}\frac{(p_1p_2\cdots p_d-1)^2}{(p_1p_2\cdots p_d)^{\ell+1}}\log\frac{1}{2^{\ell+1}}\\
			&\times  \left(
			\begin{matrix}
				1&1\\
				
			\end{matrix}
			\right)\prod_{k=1}^m\left[
			\begin{matrix}
				\frac{1}{\sqrt{2}} & \frac{1}{\sqrt{2}}  \\
				\frac{1}{\sqrt{2}} & -\frac{1}{\sqrt{2}} 
			\end{matrix}
			\right]\left[
			\begin{matrix}
				e^{\beta v_k}+e^{-\beta v_k} & 0  \\
				0 &  e^{\beta v_k}-e^{-\beta v_k}
				
			\end{matrix}
			\right]^{P_k\ell}\left[
			\begin{matrix}
				\frac{1}{\sqrt{2}} & \frac{1}{\sqrt{2}} \\
				\frac{1}{\sqrt{2}}& -\frac{1}{\sqrt{2}} 
				
			\end{matrix}
			\right] \left(
			\begin{matrix}
				1\\
				1
			\end{matrix}
			\right) \\			
			&=\sum_{\ell=1}^{\infty}\frac{(p_1p_2\cdots p_d-1)^2}{(p_1p_2\cdots p_d)^{\ell+1}}\log\frac{1}{2^{\ell}}  \prod_{k=1}^m(e^{\beta v_k}+e^{-\beta v_k})^{P_k\ell} \\
			&=\sum_{k=1}^m P_k\sum_{\ell=1}^{\infty}\frac{(p_1p_2\cdots p_d-1)^2}{(p_1p_2\cdots p_d)^{\ell+1}} \left[-\ell\log 2+\ell\log(e^{\beta v_k}+e^{-\beta v_k})\right]\\
			&=\sum_{k=1}^m P_k	\log(e^{\beta v_k}+e^{-\beta v_k})-\log 2.			
		\end{aligned}
	\end{equation*}
	\item[\bf 2.]The formula of Theorem \ref{theorem:4.1} \text{(1)} implies
	\begin{equation*}
		\begin{aligned}
			(F_{\frac{1}{2}}^{\bf w})'(\beta)=&\sum_{k=1}^m P_k \frac{v_k(e^{\beta v_k}-e^{-\beta v_k})}{e^{\beta v_k}+e^{-\beta v_k}}
		\end{aligned}
	\end{equation*}
	and
	\begin{equation*}
		\begin{aligned}
			(F_{\frac{1}{2}}^{\bf w})''(\beta)=&\sum_{k=1}^m P_k\frac{4v_k^2}{(e^{\beta v_k}+e^{-\beta v_k})^2}>0.
		\end{aligned}
	\end{equation*}
	\item[\bf 3.]The result (3) of Theorem \ref{theorem:4.1} follows from Theorem \ref{theorem:GE} and Theorem \ref{theorem:4.1} \text{(1)}, \text{(2)}.
\end{proof}

\begin{remark}
	Note that the rate function $I^{\bf w}_{\frac{1}{2}}$ is dependent only on frequency condition $(\ref{111})$, not on the dimension $d\in \mathbb{N}$ and the multiple constraint vector ${\bf p}=(p_1,...,p_d)\in \mathbb{N}^d$.
\end{remark}

Note that when $w_{\bf i}=1$ for all ${\bf i}\in \mathbb{N}^d$, $	F_{\frac{1}{2}}^{\bf w}(\beta)$ is equal to the free energy function of the sum $S_N=\sum_{i=1}^{N} \sigma_i \sigma_{2i}$ in \cite{carinci2012nonconventional}.

We say ${\bf w}=(w_{\bf i})_{{\bf i}\in \mathbb{N}^d}$ is a \emph{M\"{o}bius weight} if it is M\"{o}bius on each spin, for instance, in the spin
\begin{equation*}
	\begin{aligned}
		(\alpha_1,\alpha_2,...,\alpha_d),(\alpha p_1,\alpha_2 p_2,...,\alpha_d p_d),...,(\alpha p_1^r,\alpha_2 p_2^r,...,\alpha_d p_d^r),
	\end{aligned}
\end{equation*}
we put $w_{(\alpha_1 p_1^{i-1},\alpha_2 p_2^{i-1},...,\alpha_d p_d^{i-1})}=\mu(i)$ for all $1\leq i \leq r+1$. Then we have the following corollary
\begin{corollary}\label{coro mobi}
		For any $d\geq 1$, $p_1,p_2,...,p_d\geq 1$ with $\gcd(p_i,p_j)=1$ for all $1\leq i< j\leq d$ and ${\bf w}=(w_{{\bf i}})_{{\bf i}\in \mathbb{N}^d}$ is a M\"{o}bius weight. The free energy function associated to sum $S_{N_1\times N_2\times \cdots \times N_d}^{{\bf p,w}}$ is
	\begin{equation*}
		\begin{aligned}
			F^{\bf w}_{\frac{1}{2}}(\beta)=&\frac{6}{\pi^2}\log\left(\frac{1}{2}(e^\beta+e^{-\beta})\right).
		\end{aligned}
	\end{equation*}
\end{corollary}

\begin{proof}
	Since
	\begin{equation*}
		\begin{aligned}
			\lim_{r\rightarrow\infty}\frac{1}{r}\sum_{n=1}^{r}| \mu(n) |=\frac{6}{\pi^2},
		\end{aligned}
	\end{equation*}
	the frequency of $1,-1$ and $0$ denoted by $P_1, P_{-1}$ and $P_0$ respectively (defined in (\ref{111})), satisfying $P_{1}+P_{-1}=\frac{6}{\pi^2}$ and $P_{0}=1-\frac{6}{\pi^2}$. 
	
	Then by Theorem \ref{theorem:4.1}, we have
	\begin{equation*}
		\begin{aligned}
			F^{\bf w}_{\frac{1}{2}}(\beta)=&\sum_{k=1}^m P_k	\log(e^{\beta v_k}+e^{-\beta v_k})-\log2\\
			=&(P_0-1)\log 2+(P_1+P_{-1})\log(e^\beta+e^{-\beta})\\
			=&\frac{6}{\pi^2}\log\left(\frac{1}{2}(e^\beta+e^{-\beta})\right).
		\end{aligned}
	\end{equation*}
The proof is complete.
\end{proof}

\section{Boundary Conditions on Multiple sums}

In \cite{chazottes2014thermodynamic}, the authors consider the`multiplicative Ising model' with the parameters of the \emph{inverse temperature }$\beta$, \emph{coupling strength} $J$, and \emph{magnetic field} $h$ as the lattice spin on $\mathcal{A}^{\mathbb{N}_0}$($\mathbb{N}_0=\mathbb{N} \cup \{ 0\}$) with the Hamiltonian
	\begin{equation}\label{Hamil}
	\begin{aligned}
H(\sigma)=-\beta \left( \sum_{i\in \mathbb{N}} J\sigma_1 \sigma_2+ h\sum_{i\in \mathbb{N}}\sigma_i  \right).
	\end{aligned}
\end{equation}
In the lattice interval $[1,2N]$, they also define the Hamiltonian (\ref{Hamil}) with the boundary condition $\eta\in \{-1,1\}$ as
	\begin{equation}\label{BHamil}
\begin{aligned}
H_N^{\eta}(\sigma_{[1,2N]})=-\beta \left( \sum_{i=1}^{N} J\sigma_1 \sigma_2+ \sum_{i=1}^{2N}h\sigma_i \pm \sum_{i=N+1}^{2N}\sigma_i \eta_{2i} \right).
\end{aligned}
\end{equation}
In this circumstance, $H_N^{\emptyset}(\sigma_{[1,2N]})= \sum_{i=1}^{N} -\beta J\sigma_1 \sigma_2+ \sum_{i=1}^{2N}-\beta h\sigma_i $ denotes the Hamiltonian with free boundary conditions. 

Following \cite{chazottes2014thermodynamic}, we impose the boundary conditions on the multiple sum (\ref{14}) when $r=\frac{1}{2}$. To avoid the cumbersome computation, we restrict ourself on the case when $d=2$ since the case of general dimensions can be treated in the same fashion.

For $N_1,N_2\geq 1$, define the \emph{Dirichlet boundary condition Type 1} (BC1) by, putting $+1$ on the boundary. That is,
\begin{equation*}
	\begin{aligned}
 &\sigma_{(1,1)},...,\sigma_{(1,N_2)}=+1, \sigma_{(1,N_2)},...,\sigma_{(N_1,N_2)}=+1,\\ &\sigma_{(1,1)},...,\sigma_{(N_1,1)}=+1, \sigma_{(N_1,1)},...,\sigma_{(N_1,N_2)}=+1.
 	\end{aligned}
\end{equation*}  
The \emph{Dirichlet boundary condition Type 2} (BC2) is defined by (All $+1$ except $(i_1,i_2)$ with $1<i_1<N_1$ and $1<i_2<N_2$). The \emph{Periodic boundary condition} (BCp) is defined by (All spins have same sign on its starting point and end.)

Define the \emph{energy function} corresponding to boundary conditions by
\begin{equation}\label{FBC}
	\begin{aligned}
		F^{({\rm BCi})}(\beta):=\lim_{N_1,N_2\rightarrow\infty}\frac{1}{N_1N_2}\log\mathbb{E}_{\frac{1}{2};{\rm (BCi)}}
		\left(e^{\beta S_{N_1\times N_2}^{(p_1,p_2)}} \right),~\rm i\in\{\rm 1,2,p\}.
	\end{aligned}
\end{equation} 

\begin{theorem}\label{theorem bc1}
\item[1.]The explicit formula of energy functions $F^{({\rm BCi})}(\beta)$ for $\rm i\in\{ 1,2,p\}$ are presented as follows. 
\begin{equation*}
	\begin{aligned}
		F^{({\rm BC1})}(\beta)&=\log (e^\beta+e^{-\beta})-\log2,\\
		F^{({\rm BC2})}(\beta)&=\log(e^\beta+e^{-\beta})-\left( \frac{2p_1p_2-1}{p_1p_2}\right)\log2,\\
		F^{({\rm BCp})}(\beta)&=\log(e^\beta+e^{-\beta})-\left( \frac{2p_1p_2-1}{p_1p_2}\right)\log2+\sum_{\ell=1}^{\infty}\frac{(p_1p_2-1)^2}{(p_1p_2)^{\ell+1}}\log\left[1+\left(\frac{e^{\beta}-e^{-\beta}}{e^{\beta}+e^{-\beta}}\right)^\ell\right].
	\end{aligned}
\end{equation*}
\item[2.]The function $F^{({\rm BCi})}(\beta)$ is a differentiable with respect to $\beta \in \mathbb{R}$ for $\rm i\in\{ 1,2,p\}$.	
\item[3.]The multiple average {\rm (\ref{multiple sum average})} with boundary conditions satisfy LDP with the rate functions
\begin{equation*}
	\begin{aligned}
		I(x)=&\sup_{\beta \in \mathbb{R}}\left( \beta x-F^{({\rm BCi})}(\beta) \right),~\rm i\in\{ 1,2,p\}.
	\end{aligned}
\end{equation*}
Furthermore, if $(F^{({\rm BCi})})'(\eta)=y$, then $I(y)=\eta y- F^{({\rm BCi})}(\eta)$, $\rm i\in\{ 1,2,p\}$.	
\end{theorem}

\begin{proof}
\item[\bf 1.]Applying the decomposition in Section 3 to (\ref{FBC}), we have 
\begin{equation*}
\begin{aligned}
F^{({\rm BC1})}(\beta)&=\lim_{N_1,N_2\rightarrow\infty}\frac{1}{N_1N_2} \left\{\sum_{\ell=1}^{N_1N_2}A_{N_1,N_2;\ell}\log \frac{1}{2^{\ell+1}} \left|V^{\ell} \right|  \right.\\
&+\sum_{\ell=1}^{N_1N_2}B_{N_1,N_2;\ell}\log\frac{1}{2^{\ell+1}}\left[ (V^{\ell})_{11}+(V^{\ell})_{21}\right]    \\
&+\sum_{\ell=1}^{N_1N_2}C_{N_1,N_2;\ell}\log\frac{1}{2^{\ell+1}}\left[ (V^{\ell})_{11}+(V^{\ell})_{12}\right]    \\
&+\left.\sum_{\ell=1}^{N_1N_2}D_{N_1,N_2;\ell}\log\frac{1}{2^{\ell+1}} (V^{\ell})_{11}\right\},
\end{aligned}
\end{equation*} 	 
where the coefficients $A_{N_1,N_2;\ell}$, $B_{N_1,N_2;\ell}$, $C_{N_1,N_2;\ell}$ and $D_{N_1,N_2;\ell}$ are the numers of $\ell$ length sublatticies in $N_1\times N_2$ lattice which intersects boundary empty, exact starting point, exact end and two sides respectively.

For any $\beta \in \mathbb{R}$, we have
\begin{equation}\label{A limit}
\begin{aligned}
\lim_{N_1,N_2\rightarrow\infty}\frac{1}{N_1N_2} \sum_{\ell=1}^{N_1N_2}A_{N_1,N_2;\ell}\log \frac{1}{2^{\ell+1}} \left|V^{\ell} \right|   \leq F^{({\rm BC1})}(\beta)  
\end{aligned}
\end{equation}
and
\begin{equation}\label{K limit}
	\begin{aligned}
		 F^{({\rm BC1})}(\beta) \leq  \lim_{N_1,N_2\rightarrow\infty}\frac{1}{N_1N_2} \sum_{\ell=1}^{N_1N_2}| \mathcal{K}_{N_1\times N_2;\ell} |\log \frac{1}{2^{\ell+1}} \left|V^{\ell} \right|,  
	\end{aligned}
\end{equation} 
where $\mathcal{K}_{N_1\times N_2;\ell}$ is defined in Section 2.

Indeed, 
\begin{equation}\label{A estimate}
\begin{aligned}
| \mathcal{K}_{N_1\times N_2;\ell} |-2\left(\left\lfloor\frac{N_{1}}{p_{1}^{\ell-1}}\right\rfloor-\left\lfloor\frac{N_{1}}{p_{1}^\ell}\right\rfloor +
\left\lfloor\frac{N_{2}}{p_{2}^{\ell-1}}\right\rfloor-\left\lfloor\frac{N_{2}}{p_{2}^\ell}\right\rfloor \right)  \leq A_{N_1,N_2;\ell},
	\end{aligned}
\end{equation} 
since boundaries
\begin{equation*}
	\begin{aligned}
		(1,1),...,(1,N_2) \mbox{ and } (N_1,1),...,(N_1,N_2)
	\end{aligned}
\end{equation*}  
intersect at most $\left\lfloor\frac{N_{2}}{p_{2}^{\ell-1}}\right\rfloor-\left\lfloor\frac{N_{2}}{p_{2}^\ell}\right\rfloor$ sublatticies with length $\ell$, and boundaries  
\begin{equation*}
	\begin{aligned}
		(1,1),...,(N_1,1) \mbox{ and } (1,N_2),...,(N_1,N_2)
	\end{aligned}
\end{equation*}  
intersect at most $\left\lfloor\frac{N_{1}}{p_{1}^{\ell-1}}\right\rfloor-\left\lfloor\frac{N_{1}}{p_{1}^\ell}\right\rfloor$ sublatticies with length $\ell$.

Then (\ref{A limit}) and (\ref{A estimate}) give
\begin{equation}\label{squeeze theorem1}
	\begin{aligned}
		\lim_{N_1,N_2\rightarrow\infty}\frac{1}{N_1N_2} \sum_{\ell=1}^{N_1N_2}\left(| \mathcal{K}_{N_1\times N_2;\ell} |-2E_{N_1,N_2;\ell}\right)\log \frac{1}{2^{\ell+1}} \left|V^{\ell} \right|\leq F^{({\rm BC1})}(\beta), 
	\end{aligned}
\end{equation}
where $E_{N_1,N_2;\ell}=\left\lfloor\frac{N_{1}}{p_{1}^{\ell-1}}\right\rfloor-\left\lfloor\frac{N_{1}}{p_{1}^\ell}\right\rfloor +
\left\lfloor\frac{N_{2}}{p_{2}^{\ell-1}}\right\rfloor-\left\lfloor\frac{N_{2}}{p_{2}^\ell}\right\rfloor$.

On the other hand, combining (\ref{K limit}) and Theorem \ref{theorem:3.1}, 
\begin{equation}\label{squeeze theorem2}
	\begin{aligned}
		F^{({\rm BC1})}(\beta)   \leq  F_{\frac{1}{2}}(\beta),  
	\end{aligned}
\end{equation} 
where $F_{\frac{1}{2}}(\beta)=\log \left(\frac{1}{2}(e^\beta+e^{-\beta}) \right)$.

It remains to show 
\begin{equation}\label{remainder1}
	\begin{aligned}
		\lim_{N_1,N_2\rightarrow\infty}\frac{1}{N_1N_2} \sum_{\ell=1}^{N_1N_2}2E_{N_1,N_2;\ell}\log \frac{1}{2^{\ell+1}} \left|V^{\ell} \right|=0.
	\end{aligned}
\end{equation}
It is enough to consider (\ref{remainder1}) as
\begin{equation}\label{remainder2}
\begin{aligned}
\lim_{N_1,N_2\rightarrow\infty}\frac{C}{N_1N_2} \sum_{\ell=1}^{\infty}\ell\left( N_1(\frac{1}{p_1^{\ell-1}}-\frac{1}{p_1^{\ell}})+N_2(\frac{1}{p_2^{\ell-1}}-\frac{1}{p_2^{\ell}}) \right),
\end{aligned}
\end{equation}
where $C$ is a constant dependent only on the maximum eigenvalue of $V$. Indeed, (\ref{remainder2}) is equal to $0$ by direct computation.

By the Squeeze Theorem with (\ref{squeeze theorem1}), (\ref{squeeze theorem2}) and (\ref{remainder1}), we obtain
\begin{equation*}
\begin{aligned}
F^{({\rm BC1})}(\beta)  =\log \left(\frac{1}{2}(e^\beta+e^{-\beta}) \right).
\end{aligned}
\end{equation*}
\item[\bf 2.]Since the free energy function corresponding to (BC2) is 
\begin{equation*}
	\begin{aligned}
	F^{({\rm BC2})}(\beta)=&\lim_{N_1,N_2\rightarrow\infty}\frac{1}{N_1N_2} \left\{ \sum_{\ell=1}^{N_1N_2}A_{N_1,N_2;\ell}\log \frac{1}{2^{\ell+1}} [(V^{\ell})_{11}+(V^{\ell})_{21} ]  \right.\\
		&+\sum_{\ell=1}^{N_1N_2}B_{N_1,N_2;\ell}\log\frac{1}{2^{\ell+1}}[ (V^{\ell-1})_{11}V_{11}+(V^{\ell-1})_{21}V_{11}]   \\
		&+\sum_{\ell=1}^{N_1N_2}C_{N_1,N_2;\ell}\log\frac{1}{2^{\ell+1}}[ (V^{\ell-1})_{11}V_{11}+(V^{\ell-1})_{12}V_{21}]   \\
		&+\left.\sum_{\ell=1}^{N_1N_2}D_{N_1,N_2;\ell}\log\frac{1}{2^{\ell+1}}(V^{\ell-1})_{11}V_{11}\right\}.
	\end{aligned}
\end{equation*} 
 Similar to the proof of $F^{({\rm BC1})}$, we have
 \begin{equation*}
 	\begin{aligned}
 		F^{({\rm BC2})}(\beta)  =\lim_{N_1,N_2\rightarrow\infty}\frac{1}{N_1N_2} \sum_{\ell=1}^{N_1N_2}| \mathcal{K}_{N_1\times N_2;\ell} |\log\frac{1}{2^{\ell+1}} [(V^{\ell})_{11}+(V^{\ell})_{21} ],
 	\end{aligned}
 \end{equation*}
where
 \begin{equation*}
	\begin{aligned}
V^\ell=\left[\begin{matrix}
	\frac{1}{2}(e^\beta+e^{-\beta})^{\ell}+\frac{1}{2}(e^\beta-e^{-\beta})^{\ell} & \frac{1}{2}(e^\beta+e^{-\beta})^{\ell}-\frac{1}{2}(e^\beta-e^{-\beta})^{\ell}  \\
	\frac{1}{2}(e^\beta+e^{-\beta})^{\ell}-\frac{1}{2}(e^\beta-e^{-\beta})^{\ell} & \frac{1}{2}(e^\beta+e^{-\beta})^{\ell}+\frac{1}{2}(e^\beta-e^{-\beta})^{\ell}	
\end{matrix}\right].
 	\end{aligned}
\end{equation*}
By Lemmas \ref{lemma:2.2} and \ref{lemma:1.5}, we have
 \begin{equation*}
	\begin{aligned}
		F^{({\rm BC2})}(\beta)&=\log \left(\frac{1}{2}(e^\beta+e^{-\beta}) \right)+\sum_{\ell=1}^{\infty}\frac{(p_1p_2-1)^2}{(p_1p_2)^{\ell+1}}\log\left(\frac{1}{2}\right)\\
		&=\log \left(\frac{1}{2}(e^\beta+e^{-\beta}) \right)-\left( \frac{p_1p_2-1}{p_1p_2}\right)\log2\\
		&=\log(e^\beta+e^{-\beta})-\left( \frac{2p_1p_2-1}{p_1p_2}\right)\log2.
	\end{aligned}
\end{equation*}

\item[\bf 3.]Using the similar method as above, we obtain
\begin{equation*}
	\begin{aligned}
		F^{({\rm BCp})}(\beta)  =\lim_{N_1,N_2\rightarrow\infty}\frac{1}{N_1N_2} \sum_{\ell=1}^{N_1N_2}| \mathcal{K}_{N_1\times N_2;\ell} |\log\frac{1}{2^{\ell+1}} {\rm tr}(V^{\ell}).
	\end{aligned}
\end{equation*}
The proof is completed by the Lemma \ref{lemma:2.2}, Lemma \ref{lemma:1.5} and direct computation. 
\item[\bf 4.]Theorem \ref{theorem bc1} (2) follows from the direct calculation of the formula established in Theorem \ref{theorem bc1} (1). Theorem \ref{theorem bc1} (3) follows from Theorem \ref{theorem:GE}. This completes the proof.
\end{proof}

\begin{remark}
The reason that $F^{({\rm BC1})}$ equals (\ref{F1/2}) is the (BC1) effects at most $2(N_1+N_2)$ sublatticies and $\frac{2(N_1+N_2)}{N_1N_2}$ has limit zero as $N_1$ and $N_2$ tend to infinity. On the other hand, (BC2) and (BCp) effect almost all sublatticies in $N_1\times N_2$ lattice. This makes the difference between $F^{({\rm BC1})}$ and $F^{({\rm BCi})}$, $\rm i\in\{ 2,p\}$.
\end{remark}

\section{Conclusion and some open problems}

\subsection{Conclusion}
In Section 3, we obtain the explicit formula of the free energy function $F_r(\beta)$ associated to multiple sum (\ref{14}). Then we establish LDP of multiple average (\ref{multiple sum average}) with the rate function $I_{r}(y)=\eta y -F_r(\eta)$, if $F_r'(\eta)=y$, by the differentiability of $F_r(\beta)$ and Theorem \ref{theorem:GE}. In section 4, the LDP of the weighted multiple average (\ref{12}) is also established for $r=\frac{1}{2}$. Note that, when $r\neq \frac{1}{2}$, the free energy funciotn is difficult to compute since the matrices are not commute in general. The boundary conditions (BCi), i$\in \{\rm 1,2,p\}$ are imposed to the multiple sum (\ref{14}) in Section 5. The rigorous formulae of energy functions associated with the boundary conditions are derived. Consequently, the LDP results follow as well. The following problem remains.

\begin{problem}
In this article, we only consider the $2$-multiple sum (\ref{14}). The $k$-multiple sum is defined similarly. Namely, for ${\bf p}_1,{\bf p}_2,...,{\bf p}_{k-1} \in \mathbb{N}^d$, define
\begin{equation*}
	\begin{aligned}
		S_{N_1\times N_2\times \cdots \times N_d}^{{\bf p}_1,{\bf p}_2,...,{\bf p}_{k-1}}:=\sum_{i_i=1}^{N_1}\sum_{i_2=1}^{N_2}\cdots \sum_{i_d=1}^{N_d} \sigma_{\bf i}\sigma_{{\bf p}_1\cdot {\bf i}}\sigma_{{\bf p}_2\cdot {\bf i}}\cdots \sigma_{{\bf p}_{k-1}\cdot {\bf i}}.
	\end{aligned}
\end{equation*}	
Theorem \ref{theorem:3.1} (or Thoerem \ref{theorem:4.1}) demonstrates the absence of the phase transition phenomena of the multiplicative Ising model with 2-multiple sum. Dose the phase transition phenomena occur with respect to the $k$-multiple sum for $k\geq 3$? Such a problem depends on explicitly calculating the free energy function $F_r(\beta)$, and this is an exremely difficult task.
\end{problem}

\bibliographystyle{amsplain}


\end{document}